\documentclass{article}

\usepackage{amsmath, amsthm, amssymb, graphicx, microtype, hyperref, tikz-cd, tikz, subcaption}
\usepackage{comment}
\usepackage{geometry}
\usepackage{braket}
\usepackage{algorithm}
\usepackage{algpseudocode}
\usepackage{hyperref}
\usepackage{soul}
\usepackage{mathtools}
\usepackage{bm}
\usepackage{enumerate} 
\usepackage{ninecolors}

\tikzcdset{scale cd/.style={every label/.append style={scale=#1},
    cells={nodes={scale=#1}}}}

\newcommand{\R}{\mathbb{R}}

\newcommand{\N}{\mathbb{N}}
\newcommand{\e}{\varepsilon}

\newcommand{\step}{\operatorname{maxstep}}

\newcommand{\intr}{\operatorname{int}}

\newtheorem{defn}{Definition}
\newtheorem{thm}{Theorem}
\newtheorem{lem}{Lemma}
\newtheorem{prop}{Proposition}

\newtheorem{cor}{Corollary}

\newtheorem{prob}{Problem}
\newtheorem*{prob*}{Problem}
\newtheorem{claim}{Claim}
\newtheorem{algo}{Algorithm}

\def\ve#1{\mathchoice{\mbox{\boldmath$\displaystyle\bf#1$}}
{\mbox{\boldmath$\textstyle\bf#1$}}
{\mbox{\boldmath$\scriptstyle\bf#1$}}
{\mbox{\boldmath$\scriptscriptstyle\bf#1$}}}

\newcommand\veb{{\ve b}}
\newcommand\vecc{{\ve c}}
\newcommand\ved{{\ve d}}
\newcommand\vece{{\ve e}}

\newcommand\veg{{\ve g}}
\newcommand\veh{{\ve h}}

\newcommand\veq{{\ve q}}
\newcommand\ver{{\ve r}}

\newcommand\veu{{\ve u}}
\newcommand\vev{{\ve v}}
\newcommand\vew{{\ve w}}
\newcommand\vex{{\ve x}}
\newcommand\vey{{\ve y}}
\newcommand\vez{{\ve z}}
\newcommand\veo{{\ve 0}}
\newcommand\T{\top}
\newcommand\omo{\omega_0}
\newcommand\ome{\omega_\e}

\DeclareMathOperator{\rank}{rank}

\makeatletter
\renewcommand*{\top}{%
  {\mathpalette\@transpose{}}%
}
\newcommand*{\@transpose}[2]{%
  \scriptsize
 \raisebox{1pt}{$\m@th#1\mathsf{T}$}%
}
\makeatother

\usepackage{authblk}

\title{On the Hardness of Short and Sign-Compatible Circuit Walks}
\date{ }

\author[1]{Steffen Borgwardt}
\author[1]{Weston Grewe}
\author[2]{Sean Kafer}
\author[3]{Jon Lee}
\author[4]{Laura Sanit\`a}

\affil[1]{\small Department of Mathematical and Statistical Sciences, University of Colorado Denver}
\affil[2]{\small School of Mathematics, Georgia Institute of Technology}
\affil[3]{\small IOE Department, University of Michigan}
\affil[4]{\small Department of Computing Sciences, Bocconi University of Milan}
\begin{document}

\maketitle

\begin{abstract}

The circuits of a polyhedron are a superset of its edge directions. Circuit walks, a sequence of steps along circuits, generalize edge walks and are ``short'' if they have few steps or small total length. Both interpretations of short are relevant to the theory and application of linear programming.

We study the hardness of several problems relating to the construction of short circuit walks. We establish that for a pair of vertices of a $0/1$-network-flow polytope, it is NP-complete to determine the length of a shortest circuit walk, even if we add the requirement that the walk must be sign-compatible. Our results also imply that determining the minimal number of circuits needed for a sign-compatible decomposition is NP-complete. Further, we show that it is NP-complete to determine the smallest total length (for $p$-norms $\lVert \cdot \rVert_p$, $1 < p \leq \infty$) of a circuit walk between a pair of vertices. One method to construct a short circuit walk is to pick up a correct facet at each step, which generalizes a non-revisiting walk. We prove that it is NP-complete to determine if there is a circuit direction that picks up a correct facet; in contrast, this problem can be solved in polynomial time for TU polyhedra.

\end{abstract}

\noindent{\bf Keywords:} circuits, circuit walks, sign-compatibility, distance, complexity\\
{\bf MSC2020:}  52B05, 68Q25, 90C60

\section{Introduction}

Circuits and circuit walks generalize edges and edge walks. The \emph{circuits}, or elementary vectors \cite{g-75,r-69}, of a polyhedron are the vectors corresponding to minimal dependence relations of an underlying constraint matrix; they include all edge directions. The corresponding \emph{circuit walks} \cite{bfh-14} follow these directions, and so may pass through the interior of the polyhedron. We are interested in the complexity of constructing ``short'' circuit walks between the vertices of a polyhedron, which we interpret in two ways: a circuit walk with few steps or a circuit walk with small total length with respect to a norm. We call the total length with respect to a norm the \emph{geometric distance}. One method to construct a short circuit walk is to construct a circuit walk that enters a facet incident to the target at each step; such a walk is a generalization of a \emph{non-revisiting} edge walk. Both interpretations of short are closely related to sign-compatibility between the circuits.

In many applications, a gradual transition to a new transportation plan, schedule, database structure, or software solution is needed. Ideally, such a transition would be a short sequence of simple steps that satisfy the various constraints associated with the problem after each step---which makes the use of circuits a natural choice. We encounter such applications, for example, in the transformation of clusterings through sequences of item movements that would retain favorable properties \cite{bhz-22, bv-19a}.

First, we provide some background and notation. Then, we outline our contributions in Section \ref{sec:contributions}. Throughout, we use bold font to refer to vectors. For a vector $\vex \in \R^n$ and $1 \leq i \leq n$, $\vex(i)$ denotes the $i^{\rm th}$ coordinate of $\vex$. 
For a matrix $A \in \R^{m\times n}$ or a vector $\veb\in\R^m$, given a subset $I \subseteq [m]\coloneqq\{ 1, \ldots, m \}$, $A_I$ is the row-submatrix of $A$ whose rows are indexed by $I$ and $\veb(I)$ is the restriction of $\veb$ to the entries indexed by $I$.  Given an index $i\in[m]$, we let $A_i$ denote the (column) vector given by the $i^{\text{th}}$ row of $A$. 
 Throughout, we assume all polyhedra are rational and given in a general description $P=\{ \mathbf{x} \in \R^n \colon A \mathbf{x} = \mathbf{b}, B \mathbf{x} \leq \mathbf{d} \}$, where we allow for the possibility that $A$ is empty (in which case $\ker(A)=\R^n$).

\paragraph{Circuits.} A formal definition of the set of circuits is as follows \cite{bdf-16,bv-17,dknv-22,dhl-15}.

\begin{defn}[Circuits]\label{defn:circuits}
Let $P = \{ \mathbf{x} \in \R^n \colon A \mathbf{x} = \mathbf{b}, B \mathbf{x} \leq \mathbf{d} \}$ be a polyhedron. The set of circuits of $P$, denoted  $\mathcal{C}(A,B)=\mathcal{C}(P)$, with respect to its linear description consists of all vectors $\mathbf{g} \in \ker(A) \setminus \{ \mathbf{0} \}$,  for which $B \mathbf{g}$ is support-minimal in the set $\{ B \mathbf{x} \colon \mathbf{x} \in \ker(A) \setminus \{ \mathbf{0} \}\}$ and $\veg$ has coprime integral components.
\end{defn}

Note that $\mathcal{C}(A,B)$ depends on the matrices $A$ and $B$. We use the notation $\mathcal{C}(P)$ when matrices $A$, $B$ are clear from the context or when we assume a minimal description.  
Geometrically, circuits correspond to the one-dimensional subspaces of $\ker(A)$ obtained by intersecting it with $\ker(B')$, where $B'$ is a row-submatrix of $B$ and $\rank\binom{A}{B'} = n-1$ \cite{r-69}. Thus, the set $\mathcal{C}(P)$ contains all edge directions of $P$ and all potential edge directions as the right-hand side vectors $\veb$ and $\ved$ vary \cite{g-75}. Any form of normalization of the circuits leads to a finite set of unique representatives $\pm \veg$ for the directions. It is standard to normalize to coprime integer components, which assumes rational data.

Algebraically, circuits correspond to dependence relations for matrix $A$ with an inclusion-minimal support with respect to $B$. This support-minimality is the reason why circuits are regarded as the natural ``difference vectors'' between points: for a standard-form polyhedron $P = \{ \vex \in \R^n \colon A \vex = \veb, \vex \geq 0 \}$, the circuits are directions $\veg$ that change an inclusion-minimal set of components of a point $\vex_1 \in P$ while maintaining feasibility ($A(\vex_1 + \lambda \veg) = \veb)$. 
For a general polyhedron $P = \{ \vex \in \R^n \colon A \vex = \veb, B \vex \leq \ved \}$, the circuits are directions $\veg$ that remain ``neutral'' ($B_{I}\veg = \veo$, for maximal $I$) to as many facets as possible while maintaining feasibility.

Circuits play an important role in the theory of oriented matroids \cite{bk-84,blswz-99} and linear programming, for example, as optimality certificates \cite{g-75,o-10,st-97}. The term circuits is named after the cycles in a network, which are the minimal dependence relations in an associated graphic matroid. For many highly-structured problems of combinatorial optimization, circuits are readily interpreted in terms of the application \cite{bv-17,kps-17}. In network applications, the circuits typically correspond to cycles, paths, or cut sets in the underlying network \cite{bdfm-18,env-22} and can be used to explain classical concepts like flow decomposition \cite{amo-93}.  

\paragraph{Circuit Walks and Short Circuit Walks.}

We begin with some notation for the discussion of circuit walks \cite{bdf-16,dhl-15}. We represent a circuit walk between two given points as an ordered list of circuits and associated step lengths.

\begin{defn} \label{defn:circuit_walk}
    Let $P = \{ \mathbf{x} \in \R^n \colon A \mathbf{x} = \mathbf{b}, B \mathbf{x} \leq \mathbf{d} \}$ be a polyhedron, and let $\vex, \vey \in P$. An ordered list $W=((\veg_1, \lambda_1), 
    \ldots, (\veg_r, \lambda_r))$ is a (maximal) circuit walk from $\vex$ to $\vey$ if the following hold:
    \begin{enumerate}
        \item $\veg_i \in \mathcal{C}(P)$ for all $i\leq r$, \hfill (Circuit Direction)
        \item $\lambda_i > 0$ for all $i\leq r$, \hfill (Positive Step Length)
        \item $\vex + \sum_{i=1}^r \lambda_i \veg_i = \vey$, \hfill (Reachability)
        \item $\vex + \sum_{i=1}^k \lambda_i \veg_i \in P$ for all $1 \leq k \leq r$, \hfill (Feasibility)
        \item $(\vex + \sum_{i=1}^{k-1} \lambda_i \veg_i )+\lambda \veg_k \notin P$ for all $1 \leq k \leq r$ and $\lambda > \lambda_k$. \hfill (Maximality)
    \end{enumerate}
    Each $(\veg_i, \lambda_i)$ is a called a step.
\end{defn}
We emphasize that, in this work, the term ``walk" necessitates maximal step length at each step.
Often, we are interested in the point resulting from a maximal step in direction $\veg$ from a point $\vev \in P$; we let $\step(\vev,\veg; P)$ denote this resulting point. When $P$ is clear from context, we simply write $\step(\vev,\veg)$. We say a circuit $\veg$ is \emph{feasible} at $\vex$ if there exists $\e > 0$ such that $\vex + \e \veg \in P$. 

\autoref{defn:circuit_walk} allows a formalization of some natural notions for what makes a walk ``short'': 
the cardinality $\lvert W \rvert = r$ determines the number of steps and the value $\sum_{i=1}^r \lVert \lambda_i \veg_i \rVert$ for some norm $\|\cdot\|$ determines the total length with respect to $\|\cdot\|$. Short circuit walks relate to the studies of the {\em circuit distance} and the {\em circuit diameter}, as well as the construction of {\em conformal sums} and {\em sign-compatible circuit walks}.

Circuit distances and diameters \cite{bfh-14} generalize combinatorial distances and diameters. The circuit distance from vertex $\vex \in P$ to vertex $\vey \in P$ is the smallest number $r= \lvert W \rvert$ such that there exists a circuit walk $W = ((\veg_1, \lambda_1), \ldots, (\veg_r, \lambda_r))$ from $\vex$ to $\vey$. Note that, unlike edge walks, circuit walks are not always reversible---the reversed list may not have maximal step lengths. Thus, the distance from $\vex$ to $\vey$ might not equal the distance from $\vey$ to $\vex$. The circuit diameter of $P$ is the largest circuit distance between any pair of vertices of $P$. The circuit analogue of the famous Hirsch conjecture, the {\em circuit diameter conjecture} \cite{bfh-14}, asks whether there always exists a circuit walk of at most $f-d$ steps for a $d$-dimensional polyhedron with $f$ facets. Unlike the disproved Hirsch conjecture \cite{kw-67,s-11}, the circuit diameter conjecture is open. 

The studies of circuit distances, diameters, and the associated conjecture are motivated in several ways. Clearly, the circuit diameter is a lower bound on the combinatorial diameter, and thus is studied as a proxy. Just as the combinatorial diameter relates to the possible efficiency of a primal Simplex method, the circuit diameter relates to the possible efficiency of a {\em circuit augmentation scheme} for solving a linear program \cite{bv-19b,bv-19c,dhl-15,env-22,gdl-15}. Further, a resolution of the circuit diameter conjecture would give insight to the reason the Hirsch conjecture does not hold: whether it is the restriction to edges or the restriction to maximal step lengths \cite{bbb-23,bsy-18}. 

Given two vertices $\vex$ and $\vey$, one method to construct a short circuit walk is to construct a walk that enters a facet incident to $\vey$ at each step, and does not leave it in any subsequent step. Such a walk necessarily has length bounded by the dimension. We call circuit walks with this property ``direct.'' This property is reminiscent of the \emph{non-revisiting property} used to analyze edge walks in polyhedra. In fact, the associated non-revisiting conjecture is equivalent to the Hirsch conjecture \cite{kw-67}. 

The classical concepts of sign-compatibility and conformal sums relate closely to the notion of a direct circuit walk. Two vectors $\vex, \vey \in \R^n$ are {\em sign-compatible} if they belong to the same orthant of $\R^n$, that is, $\vex(i) \cdot \vey(i) \geq 0$ for $i=1,...,n$. We say that $\vex$ and $\vey$ are {\em sign-compatible with respect to matrix $B$} if the corresponding vectors $B \vex$ and $B \vey$ are sign-compatible; we drop the reference to $B$ when it is clear from context. 

Let $\vex_1, \ldots, \vex_r \in \R^n$ and $\lambda_1, \ldots, \lambda_r \in \R$. The sum $\vex = \sum_{i=1}^r \lambda_i \vex_i$ is a \emph{conformal sum} (with respect to a matrix $B \in \R^{m\times n}$) if $\lambda_i > 0$ for each $1 \leq i \leq r$ and each pair $\vex_i$, $\vex_j$ is sign-compatible (w.r.t. $B$). The set of circuits is the unique, inclusion-minimal set of directions that satisfies the so-called {\em conformal sum property} \cite{g-75,o-10}: for any polyhedron $P = \{ \vex \in \R^n \colon A \vex = \veb, B \vex \leq \ved \}$, any vector $\veu \in \ker(A)$  can be written as a {\em conformal sum} of circuits with respect to $B$. A conformal sum corresponds to a sequence of (non-maximal) sign-compatible circuit steps in which the steps can be applied {\em in any order}. For vertices $\vex, \vey \in P$ and $\veu=\vey-\vex$, we call a decomposition of $\veu$ into a conformal sum of circuits a \emph{sign-compatible circuit decomposition}. We note that any order of steps leads to a sequence of feasible points in $P$; formally, $\vex + \sum_{i \in I} \lambda_i\veg_i \in P$ for every $I \subseteq \{ 1, \ldots, t \}$. A sign-compatible circuit decomposition does not necessarily have maximal step lengths, and thus it may not be a \textit{walk}: it is only guaranteed to satisfy the first four properties in \autoref{defn:circuit_walk}. The points of a sign-compatible circuit decomposition lie in a `narrow corridor' contained in $P$: the intersection of two cones emanating from $\vex$ and $\vey$. Each cone is formed from parallel copies of all facets of $P$ translated to contain $\vex$ or $\vey$, respectively, and the respective halfspaces are chosen to contain $\vey$ or $\vex$, respectively. \cite{bbb-23,bv-17}. \autoref{figure:SC-Cones} displays an example. We note that neither direct circuit walks nor sign-compatible circuit walks are guaranteed to exist \cite{bdf-16}. 

Sign-compatible circuit decompositions and walks $W$ readily lend themselves to an interpretation as being short through a small total length $\sum_{i=1}^r \lVert \lambda_i \veg_i \rVert$. In fact, given $P = \{ \vex : A\vex = \veb, B \vex \leq \ved \}$, \emph{any} sign-compatible $W$ is a minimizer of $\sum_{i=1}^r \lVert \lambda_i B\veg_i \rVert_1$, which is the 1-norm $\| B(\cdot) \|_1$ with respect to $B$. When $P$ is in standard form, any sign-compatible $W$ is a minimizer of the $1$-norm. We will be able to exhibit that the computation of a circuit decomposition or walk $W$ that minimizes $\sum_{i=1}^r \lVert \lambda_i \veg_i \rVert_p$ for $p>1$ is hard, regardless of whether $W$ is restricted to sign-compatibility or not.
Due to the arbitrary ordering of the steps in a sign-compatible decomposition, repeated use of a circuit can be avoided, which leads to the existence of a decomposition of at most $n - \rank(A)$ steps. Such a decomposition can be computed efficiently \cite{bv-19c}. These properties have made them a valuable tool in the study of circuit diameters and augmentation \cite{bbb-23,dknv-22}.

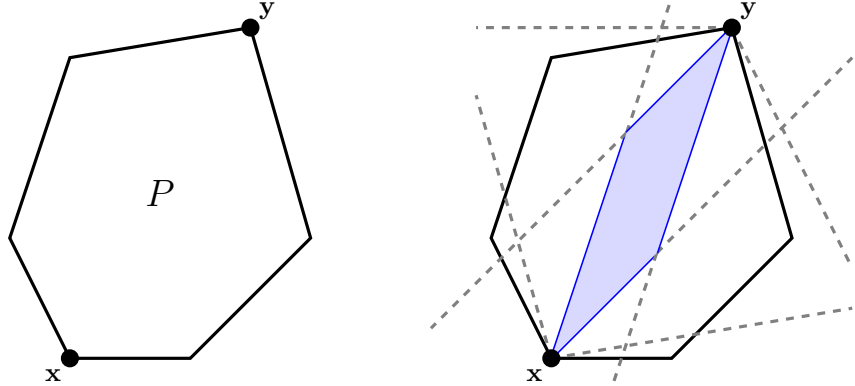
\begin{figure}
    \centering
    \begin{tikzpicture}[scale=.8, inner sep=.75mm, minicirc/.style={circle,draw=black,fill=black,thick}]
        \node (circ1) at (-2,0) [minicirc] {};
        \node (circ1) at (1,5.5) [minicirc] {};
        \draw[very thick, black] (-2,0) -- (-3,2) -- (-2,5) -- (1,5.5) -- (2,2) -- (0,0) -- cycle;
        \node at (-2.28,-.28) {$\vex$};
        \node at (1.28,5.78) {$\vey$};
        \node at (-.5,2.75) {\Large $P$};

        \draw[very thick, black] (6,0) -- (5,2) -- (6,5) -- (9,5.5) -- (10,2) -- (8,0) -- cycle;
        \node at (5.72,-.28) {$\vex$};
        \node at (9.28,5.78) {$\vey$};

        \draw[very thick, dashed, gray] (6,0) -- (11, 5/6);
        \draw[very thick, blue] (6,0) -- (7.75, 1.75);
        \draw[very thick, dashed, gray] (7.75, 1.75) -- (11, 5);
        \draw[very thick, blue] (6,0) -- (7.25,3.75);
        \draw[very thick, dashed, gray] (7.25,3.75) -- (8,6);
        \draw[very thick, dashed, gray] (6,0) -- (4.75,3.5+3.5/4);

        \draw[very thick, dashed, gray] (9,5.5) -- (4.75,5.5);
        \draw[very thick, blue] (9,5.5) -- (7.25,3.75);
        \draw[very thick, dashed, gray] (7.25,3.75) -- (4,.5);
        
        \draw[very thick, blue] (9,5.5) -- (7.75, 1.75);
        \draw[very thick, dashed, gray] (7.75,1.75) -- (7,-.5);
        \draw[very thick, dashed, gray] (9,5.5) -- (11,1.5);

        \fill[blue!15] (6,0) -- (7.75,1.75) -- (9,5.5) -- (7.25,3.75) -- cycle;
        \node (circ1) at (6,0) [minicirc] {};
        \node (circ1) at (9,5.5) [minicirc] {};
    \end{tikzpicture}
    \caption{A polytope $P$ with two vertices $\vex, \vey$ (left). The `narrow corridor' of sign-compatible circuit decompositions of $\veu = \vey - \vex$ (right).}
    \label{figure:SC-Cones}
\end{figure}

\subsection{Contributions}\label{sec:contributions}

We analyze the complexity of various problems relating to the construction of short circuit walks and show that many such problems are NP-hard in general. Additionally, we provide some examples where constructing a short circuit walk becomes efficient. 

We first consider the problem of determining the circuit distance between a pair of vertices of a polytope, formally defined as follows.\\

\begin{prob}[\textsc{Circ-Dist}]
Let $P = \{ \vex : A\vex = \veb, \, B\vex \leq \ved \}$ be a rational polyhedron, let $\vex, \vey \in P$ be vertices, and let $k \in \mathbb{N}$. Determine if there exists a list $W =  ((\veg_1, \lambda_1), \ldots, (\veg_r, \lambda_r))$ such that $W$ is a circuit walk from $\vex$ to $\vey$ with $r \leq k$. 
\end{prob} 

The above problem is already known to be NP-hard for the perfect matching polytope on a bipartite graph \cite{cs-23}. We here give a different hardness proof, which applies to $0/1$-circulation polytopes. Besides extending the result to this family of polytopes, our alternative proof allows us to infer hardness of two other problems, related to sign-compatible circuit decompositions and walks. Specifically, we formulate the following problems: 

\begin{prob}[\textsc{SC-Decomp-Dist}]\label{prob:sc_decomp}
    Let $P = \{ \vex : A\vex = \veb, \, B\vex \leq \ved \}$ be a rational polyhedron, let $\vex, \vey \in P$ be vertices, and let $k \in \mathbb{N}$. Determine if there exists a list $W=((\veg_1, \lambda_1), \ldots, (\veg_r, \lambda_r))$ such that $W$ is a sign-compatible circuit decomposition of $\vey - \vex$ with respect to $B$ with $r \leq k$. 
\end{prob} 

\begin{prob}[\textsc{SCM-Circ-Dist}]\label{prob:scm_dist}
    Let $P = \{ \vex : A\vex = \veb, \, B\vex \leq \ved \}$ be a rational polyhedron, $\vex, \vey \in P$ be vertices, and $k \in \mathbb{N}$. Determine if there exists a list $W=((\veg_1, \lambda_1), 
    \ldots, (\veg_r, \lambda_r))$ such that $W$ is a circuit walk from $\vex$ to $\vey$, $r \leq k$ and $W$ is a sign-compatible circuit decomposition of $\vey-\vex$ with respect to $B$. 
\end{prob} 

It is easy to determine if two vertices differ by a single circuit. Therefore, the hardness of \textsc{Circ-Dist}, \textsc{SC-Decomp-Dist}, and \textsc{SCM-Circ-Dist} becomes interesting when $k \geq 2$. We prove \textsc{Circ-Dist}, \textsc{SC-Decomp-Dist} and \textsc{SCM-Circ-Dist} are NP-hard, even when $k=2$, $P$ is a circulation polytope on an Eulerian digraph with in-degree and out-degree $2$, and each edge of the digraph has capacity bound $0 \leq \vex(e) \leq 1$ (\autoref{cor:2step-nphard}). To do so, we prove \autoref{thm:2stephard}, which relates the circuit distance to the existence of Hamiltonian cycles in the underlying graph. In this setting, the flow created by sending one unit of flow across every edge is a vertex; we call this vertex the \emph{full flow}.

\begin{thm}
Let $G$ be an Eulerian digraph where each node has in-degree and out-degree $2$. Let $P$ be the $0/1$-circulation polytope on $G$, let $\mathbf{0} \in P$ be the zero flow, and let $\vex \in P$ be the full flow. The circuit distance from $\mathbf{0}$ to $\vex$ is $2$ if and only if $G$ can be decomposed into two simple Hamiltonian dicycles. 
\end{thm}

\autoref{thm:2stephard} is applied to show hardness of \textsc{Circ-Dist} (\autoref{cor:2step-nphard}), \textsc{SCM-Circ-Dist} (\autoref{cor:scm-2step-hard}), and \textsc{SC-Decomp-Dist} (\autoref{cor:scm-decomp-hard}). polyhedra. When the value of $k$ is polynomial in the size of the input, \textsc{Circ-Dist}, \textsc{SC-Decomp-Dist}, and \textsc{SCM-Circ-Dist} are NP-complete. The circulation polytopes belong to the class of network-flow polytopes, which are arguably the most important subclass of \emph{totally-unimodular} polyhedra. Recall that a matrix $A$ is totally-unimodular (TU) if every square, nonsingular submatrix of $A$ has determinant $\pm 1$. A polyhedron $P$ is a \emph{TU polyhedron} if there exist matrices $A$ and $B$ such that $\binom{A}{B}$ is TU and vectors $\veb$, $\ved$ such that $P = \{ \vex : A\vex = \veb, B \vex \leq \ved \}$. The circuits of network-flow polyhedra are well-understood: each circuit is a (simple) cycle in the underlying undirected graph. Because network-flow polyhedra are TU polyhedra, each circuit has components in $\{ 0, \pm 1 \}$. For a $0/1$-network-flow polytope, every circuit walk is a \emph{vertex walk}: the resulting point after each maximal step is a vertex. Despite this simplicity, short circuit walks as described in Problems 1, 2, and 3 are still hard to find.

Our second notion of a short circuit walk is a circuit walk whose total length is minimized for a given norm. We transfer the same example, a $0/1$-circulation polytope on an Eulerian digraph with in-degree and out-degree $2$, to show that determining the minimal total length of a circuit walk is NP-hard. We formulate the problem as follows.

\begin{prob*}[$\lVert \cdot \rVert$\textsc{-Dist}]
    Let $P = \{ \vex : A\vex = \veb, \, B\vex \leq \ved \}$ be a rational polyhedron, $\vex, \vey \in P$ be vertices, and $\lVert \cdot \rVert$ be a norm on $P$. Determine the minimum of $\sum_{i=1}^k \lVert \lambda_i \veg_i \rVert$ where $((\veg_1, \lambda_1), 
    \ldots, (\veg_r, \lambda_r))$ is a circuit walk from $\vex$ to $\vey$.
\end{prob*}

We show that $\lVert \cdot \rVert$\textsc{-Dist} is NP-hard for the norms $\lVert \cdot \rVert_p$ with $1 < p \leq \infty$ (Lemmas \ref{lem:p-norm} and \ref{lem:sup-norm}), even for network-flow polytopes. Given $1 < p < \infty$, we show that the minimum length between a particular pair of vertices, with respect to $\lVert \cdot \rVert_p$, is $2\sqrt[p]{n}$ if and only if the underlying digraph can be decomposed into two Hamiltonian cycles. Given $p = \infty$, we show that the minimum length between a particular pair of vertices, with respect to $\lVert \cdot \rVert_p$, is $2$ if and only if the underlying digraph can be decomposed into two Hamiltonian cycles.

In Section \ref{sec:single_step_hard}, we study the complexity of constructing a direct circuit walk. In particular, if each step of the walk picks up a new, correct facet, then the number of steps of the walk is bounded by the dimension of the polyhedron. This first motivates studying the hardness of finding a circuit step that enters a particular facet. We formulate the problem as follows.

\begin{prob*}[\textsc{Facet-Step}]
Given a rational polyhedron $P = \{ \vex : A\vex = \veb, \, B \vex \leq \ved \}$, a facet $F$ of $P$ defined by the tight inequality $B^\T_i \vex = \ved(i)$,
and a point $\veu \in P$, compute a circuit $\veg$ such that $\step(\veu, \veg) \in F$ or determine that no such $\veg$ exists. 
\end{prob*}

It is known to be NP-hard to find a circuit step that enters a particular face $F$ of a polytope $P$ with dimension $\operatorname{dim}(F) \leq \operatorname{dim}(P)-2$ \cite{dks-22}; we generalize this result to prove that \textsc{Facet-Step} is NP-hard in Subsection \ref{subsec:facet-step}. 

\begin{thm}
    \textsc{Facet-Step} is NP-complete.
\end{thm}

We are then interested in the more general problem of finding a circuit step that enters \textit{any} facet incident to a target vertex. Given a vertex $\vew \in P$, we define $\mathcal{F}(\vew)$ to be the union of facets of $P$ incident to $\vew$. We formulate the problem as follows.

\begin{prob*}[\textsc{Incident-Facet-Step}]
Given a rational polyhedron $P = \{ \vex : A\vex = \veb, \, B\vex \leq \ved \}$ and vertices $\vev, \vew$ of $P$, compute a circuit $\veg$ such that $\step(\vev, \veg) \in \mathcal{F}(\vew)$ or determine that no such $\veg$ exists. 
\end{prob*}

In Subsection \ref{subsec:incident-facet}, we prove that \textsc{Incident-Facet-Step} is NP-complete. To do so, we modify a hard instance of \textsc{Facet-Step}. The modification is chosen carefully so that a solution to \textsc{Incident-Facet-Step} is a solution to the unmodified instance of \textsc{Facet-Step}.

\begin{thm}
    \textsc{Incident-Facet-Step} is NP-complete.
\end{thm}

\textsc{Incident-Facet-Step} is related to the problem \textsc{SCM-Step}, which asks if there exists a maximal circuit step that can appear as the first step in a sign-compatible circuit decomposition. While these two problems are not equivalent, a positive answer to \textsc{SCM-Step} is a positive answer to \textsc{Incident-Facet-Step}. At present, our methods do not resolve the hardness of \textsc{SCM-Step} in general. In Section \ref{sec:TU_easy}, we investigate cases where \textsc{Incident-Facet-Step} and \textsc{SCM-Step} are easy. 

While Sections \ref{sec:short_walk_hard} and \ref{sec:single_step_hard} are primarily focused on showing several problems are hard, in Section \ref{sec:TU_easy}, we address cases where finding short circuit walks is easy. For every polyhedron $P = \{ \vex: A\vex = \veb, \, B\vex \leq \ved \}$, there is an associated polyhedron $P_{A,B}$ called the \emph{polyhedral model of circuits} \cite{bv-19c}. The vertices of $P_{A,B}$ are of form $(\vex, \vey^+, \vey^-)$, and if $\vex \neq \mathbf{0}$, then $\vex$ is a circuit of $P$. We demonstrate with \autoref{algo:TU_facet_step} how to use $P_{A,B}$ for a TU polyhedron $P$ to efficiently solve \textsc{Incident-Facet-Step}. 

\begin{thm}
    \textsc{Incident-Facet-Step} can be solved in polynomial time for TU polyhedra.
\end{thm}

The proof of correctness (\autoref{lem:TU-Alg-Poly}) crucially uses the fact that all circuits of TU polyhedra have components in $\{ 0, \pm 1 \}$.  
Because all vertices corresponding to sign-compatible circuits are on a shared face of $P_{A, B}$
\cite{bv-19c}, we are able to further show that the problem can be solved efficiently for TU polyhedra.

Additionally, we provide graph-theoretic algorithms  to solve \textsc{Incident-Facet-Step} and \textsc{SCM-Step} for network-flow polyhedra. \autoref{algo:NF_facet_incident} gives an approach using the residual network with respect to a flow to determine the solution to \textsc{Incident-Facet-Step}. The algorithm essentially finds a cycle in a modified version of the residual network. To find a sign-compatible circuit step, we show that we only need to adjust the construction of the modified residual network: some edges must be deleted and some capacities changed.

Section \ref{sec:TU_easy} concludes by demonstrating that the class of $(n,d)$-parallelotopes, a generalization of parallelotopes, has the property that there exists a sign-compatible circuit walk between every pair of vertices. Thus, they represent a class of yes-instances for \textsc{Incident-Facet-Step} and \textsc{SCM-Step} regardless of the vertices chosen. In fact, under some additional assumptions, a simple polyhedron with a sign-compatible circuit walk between each pair of vertices is an $(n, \, d)$-parallelotope.

\begin{thm}
    Let $P$ be a simple, pointed polyhedron. 
    If for all pairs of vertices $\vev$, $\vew \in P$, every sign-compatible circuit decomposition of $\vew-\vev$ gives a sign-compatible circuit walk of length $d$, where $d$ is the dimension of the minimal face containing $\vev$ and $\vew$,
    then $P$ is an $(n,d^*)$-parallelotope, where $d^*$ is the maximum of $d$ over all choices of $\vev$ and $\vew$. 
\end{thm}

We conclude in Section \ref{sec:conclusion} with a discussion on related open questions. This includes, but is not limited to, the hardness of determining the circuit diameter, the hardness of determining if there exists a sign-compatible circuit decomposition with length dimension minus 1, and a classification of polyhedra that have sign-compatible circuit walks between every pair of vertices. 

\setcounter{thm}{0}

\section{Determining the Shortest Circuit Walk is Hard}\label{sec:short_walk_hard}

In this section, we show that it is hard to construct the shortest circuit walks for either notion of short: a circuit walk with the fewest steps or a circuit walk with minimal total length. In particular, we first show in Section \ref{subsec:cd_hard} that the problems \textsc{Circ-Dist}, \textsc{SCM-Circ-Dist}, and \textsc{SC-Decomp-Dist} are hard.
Further, in Section \ref{subsec:cd_complete} we prove that the problems are NP-complete when the number of steps is bounded by the size of the problem input. In Section \ref{subsec:min_length_hard}, we show that it is NP-hard to determine the minimal total length of a circuit walk. In Section \ref{subsec:addtl_remarks}, we state some additional remarks on approximability.

\subsection{Hardness of Circuit Distance}\label{subsec:cd_hard}

First, we prove \textsc{Circ-Dist} is NP-hard, and as corollaries we show \textsc{SCM-Circ-Dist} and \textsc{SC-Decomp-Dist} are NP-hard. We will prove hardness of \textsc{Circ-Dist} by reduction to the \textsc{k-Hamiltonian-Path} problem (which asks if  an Eulerian digraph can be decomposed into $k$ simple Hamiltonian dicycles), a known NP-complete problem \cite{p-84}, even for an Eulerian digraph with in-degree and out-degree $2$ at each node \cite{p-84}. We assume this restricted setting in our proofs. This setting and its undirected variant, $4$-regular graphs, have been instrumental in establishing the hardness of cycle-finding and tree decomposition problems \cite{ach-21,ipst-23}.

We begin with some background on network-flow polytopes. For a directed graph $G$, a set of supplies and demands $\veb(i)$ for each vertex $i$, and a set of costs $\vecc(j)$ and capacities $\veu(j)$ for each edge $j$, the \emph{network-flow problem} is modeled as the linear program
\[
    \min \{ \vecc^\T\vex : A_G\vex = \veb, \mathbf{0} \leq \vex \leq \veu \},
\]
where $A_G$ is the oriented node-arc incidence matrix for $G$. The vector $\vex$ is interpreted as the amount of flow sent on each edge. We call the underlying polytope the \emph{network-flow polytope}, which consists of all feasible flows that satisfy the supply and demand constraints. Given $\vex' \in P$, we make use of the \textit{residual network} $G(\vex')$ which is constructed as follows (see \cite{amo-93}). The supply/demand at each node is replaced with a $0$. The capacities $(0, \veu(e))$ on each edge $e$ are replaced with the capacities $(0, \veu(e) - \vex(e))$. Finally, for each edge $e = (a,b)$ with $\vex(e) > 0$, we include ``reversed'' edges $e' = (b,a)$ with capacities $(0, \vex(e))$. See \autoref{figure:residual_net_ex} for an example. The circuits of the network-flow polytope are characterized in \autoref{prop:circ-char}.

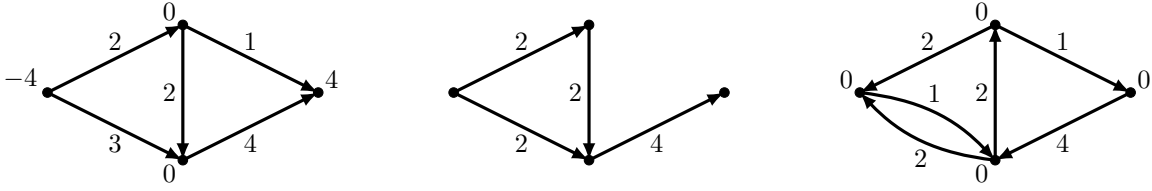
\begin{figure}
    \centering
    \begin{tikzpicture}[scale=.9]
        \draw[very thick, -latex] (0, 0) -- (2, 1);
        \draw[very thick, -latex] (0, 0) -- (2, -1);
        \draw[very thick, -latex] (2, 1) -- (2, -1);
        \draw[very thick, -latex] (2, 1) -- (4, 0);
        \draw[very thick, -latex] (2, -1) -- (4, 0);

        \draw[fill=black] (0, 0) circle (2pt);
        \draw[fill=black] (2, 1) circle (2pt);
        \draw[fill=black] (2, -1) circle (2pt);
        \draw[fill=black] (4, 0) circle (2pt);

        \node at (-.4, .2) {$-4$};
        \node at (1.8, 1.2) {$0$};
        \node at (1.8, -1.2) {$0$};
        \node at (4.2, .2) {$4$};

        \node at (1, .75) {$2$};
        \node at (1, -.75) {$3$};
        \node at (3, .75) {$1$};
        \node at (3, -.75) {$4$};
        \node at (1.8, 0) {$2$};

        \draw[very thick, -latex] (6, 0) -- (8, 1);
        \draw[very thick, -latex] (6, 0) -- (8, -1);
        \draw[very thick, -latex] (8, 1) -- (8, -1);
        \draw[very thick, -latex] (8, -1) -- (10, 0);

        \draw[fill=black] (6, 0) circle (2pt);
        \draw[fill=black] (8, 1) circle (2pt);
        \draw[fill=black] (8, -1) circle (2pt);
        \draw[fill=black] (10, 0) circle (2pt);

        \node at (7, .75) {$2$};
        \node at (7, -.75) {$2$};
        \node at (9, -.75) {$4$};
        \node at (7.8, 0) {$2$};

        \draw[very thick, latex-] (12, 0) -- (14, 1);
        \draw[very thick, latex-] (12, 0) to[bend right=20] node[below] {$2$} (14, -1);
        \draw[very thick, -latex] (12, 0) to[bend left=20] node[above] {$1$} (14, -1);
        \draw[very thick, latex-] (14, 1) -- (14, -1);
        \draw[very thick, -latex] (14, 1) -- (16, 0);
        \draw[very thick, latex-] (14, -1) -- (16, 0);

        \draw[fill=black] (12, 0) circle (2pt);
        \draw[fill=black] (14, 1) circle (2pt);
        \draw[fill=black] (14, -1) circle (2pt);
        \draw[fill=black] (16, 0) circle (2pt);

        \node at (11.8, .2) {$0$};
        \node at (13.8, 1.2) {$0$};
        \node at (13.8, -1.2) {$0$};
        \node at (16.2, .2) {$0$};

        \node at (13, .75) {$2$};
        \node at (15, .75) {$1$};
        \node at (15, -.75) {$4$};
        \node at (13.8, 0) {$2$};
    \end{tikzpicture}
    \caption{Construction of a residual network. Original network (left), feasible flow (center) and residual network (right).}
    \label{figure:residual_net_ex}
\end{figure}

\begin{prop} \label{prop:circ-char}
The circuits of a network-flow polytope are  the simple cycles of the underlying undirected graph. The feasible circuits at a flow $\vex'$ are the simple dicycles in the underlying residual network $G(\vex')$.
\end{prop}

The circuits of the network-flow polytope are described as simple cycles of the underlying graph \cite{m-60, m-66, r-69}. In vector representation, each entry $\veg(i)$ of a circuit $\veg$ corresponds to an edge of the graph, with $\veg(i) = 1$ if flow is sent along an edge in the direction of its orientation, $\veg(i) = -1$ if flow is sent along an edge in the opposite direction of its orientation, and $\veg(i) = 0$ if no flow is sent along the edge.

The \emph{circulation problem} is the special case of the network-flow problem where the excess supply or demand at each node is $0$. Throughout, we consider a circulation polytope where $\veu = \mathbf{1}$. We call this polytope the \emph{$0/1$-circulation polytope} because all vertices $\vev$ satisfy $\vev \in \{ 0,1 \}^n$. We will show that \textsc{Circ-Dist} is NP-hard even for the $0/1$-circulation polytope. For any circulation polytope, the zero vector $\mathbf{0}$ is a vertex. When $G$ is Eulerian, the flow created by sending a unit across every edge is a vertex of the associated $0/1$-circulation polytope; we call this vertex the \emph{full flow}. To show that \textsc{Circ-Dist} is NP-hard, we will show that there exists a circuit walk of length $2$ from $\veo$ to the full flow if and only the underlying graph can be decomposed into two Hamiltonian cycles.  \autoref{figure:hamiltonian-decomp} provides an illustration for our decomposition used in the proof.

\begin{thm} \label{thm:2stephard}
Let $G$ be an Eulerian digraph where each node has in-degree and out-degree $2$. Let $P$ be the $0/1$-circulation polytope on $G$, let $\mathbf{0} \in P$ be the zero flow, and let $\vex \in P$ be the full flow. The circuit distance from $\mathbf{0}$ to $\vex$ is 2 if and only if $G$ can be decomposed into two simple Hamiltonian dicycles.  
\end{thm}

\begin{proof}
\autoref{prop:circ-char} implies that a circuit can have at most one incoming edge and at most one outgoing edge for each vertex of the graph. Because each vertex has two outgoing edges in the flow, the circuit distance is at least $2$. 

The reverse direction is immediate. Suppose that $G$ can be decomposed into two Hamiltonian dicycles. Each Hamiltonian dicycle defines a circuit of the polytope. Each circuit is feasible and has unit step length. Thus, $\vex$ is the sum of the two circuits.
Therefore, the circuit distance from $\mathbf{0}$ to $\vex$ is $2$.

Now, we assume the circuit distance from $\mathbf{0}$ to $\vex$ is $2$. We will show that $G$ can be decomposed into two simple Hamiltonian dicycles. 
Because each edge of the graph has positive flow for $\vex$ and zero flow for $\mathbf{0}$, it follows that each edge appears in at least one of the circuits. Let $((\veg_1, \lambda_1), \, (\veg_2, \lambda_2))$ be the circuit walk from $\mathbf{0}$ to $\vex$. By \autoref{prop:circ-char}, $\veg_1$ is a simple dicycle in $G$. As every edge has upper capacity $1$ and this is a step from $\mathbf{0},$ it follows that $\lambda_1 = 1$ due to the maximality requirement for the step. Likewise, $\lambda_2 = 1$. We set $\vex' = \mathbf{0} + \veg_1.$ By construction, it follows that $\vex = \vex' + \veg_2$.
Again by \autoref{prop:circ-char}, $\veg_2$ is a simple dicycle of $G(\vex')$, the residual network at $\vex'$. We claim $\veg_2$ is a simple dicycle of $G$. Becasue $\veu(j)=1$ for each $j$ and the step length of $\veg_1$ is $1$, we have that $G(\vex')$ is the graph resulting from $G$ by reversing the orientation of  every edge of $\veg_1$. It follows that no edge in $\veg_1$ can appear in $\veg_2$, because those edges are not in the residual network. Therefore, $\veg_2$ is a simple dicycle consisting of edges that appeared in the original graph $G.$ Thus, $\veg_2$ is a simple dicycle in $G.$

We claim $\veg_1$ and $\veg_2$ are both Hamiltonian dicycles: each vertex of $G$ has 2 incoming edges, and each edge must appear in one of the circuits. It cannot be that both incoming edges appear in the same circuit; otherwise such a circuit would not be a cycle. Therefore, each $\veg_i$ contains an incoming edge for each vertex. Thus, each $\veg_i$ is a Hamiltonian dicycle. 
\end{proof}

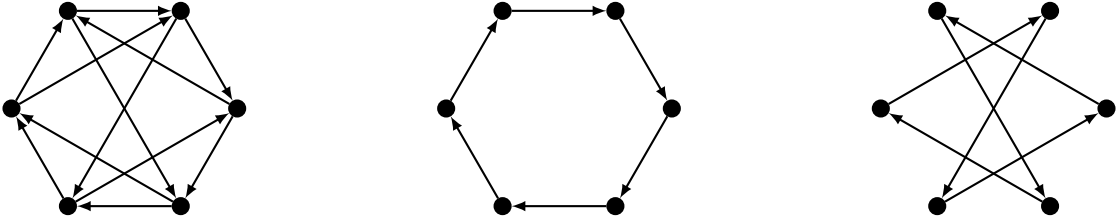
\begin{figure}
    \centering
    \begin{subfigure}[b]{0.3\textwidth}
        \centering
        \begin{tikzpicture}[scale=1.5, inner sep=.75mm, minicirc/.style={circle,draw=black,fill=black,thick}]

            \node (circ1) at ( 60:1) [minicirc] {};
            \node (circ2) at (120:1) [minicirc] {};
            \node (circ3) at (180:1) [minicirc] {};
            \node (circ4) at (240:1) [minicirc] {};
            \node (circ5) at (300:1) [minicirc] {};
            \node (circ6) at (360:1) [minicirc] {};

            \draw[black, thick, latex-] (circ1) -- (circ2);
            \draw[black, thick, latex-] (circ2) -- (circ3);
            \draw[black, thick, latex-] (circ3) -- (circ4);
            \draw[black, thick, latex-] (circ4) -- (circ5);
            \draw[black, thick, latex-] (circ5) -- (circ6);
            \draw[black, thick, latex-] (circ6) -- (circ1);

            \draw[black, thick, latex-] (circ1) -- (circ3);
            \draw[black, thick, latex-] (circ3) -- (circ5);
            \draw[black, thick, latex-] (circ5) -- (circ2);
            \draw[black, thick, latex-] (circ2) -- (circ6);
            \draw[black, thick, latex-] (circ6) -- (circ4);
            \draw[black, thick, latex-] (circ4) -- (circ1);
        \end{tikzpicture}
    \end{subfigure}
    \hfill
    \begin{subfigure}[b]{0.3\textwidth}
        \centering
        \begin{tikzpicture}[scale=1.5, inner sep=.75mm, minicirc/.style={circle,draw=black,fill=black,thick}]
            \node (circ1) at ( 60:1) [minicirc] {};
            \node (circ2) at (120:1) [minicirc] {};
            \node (circ3) at (180:1) [minicirc] {};
            \node (circ4) at (240:1) [minicirc] {};
            \node (circ5) at (300:1) [minicirc] {};
            \node (circ6) at (360:1) [minicirc] {};  

            \draw[black, thick, latex-] (circ1) -- (circ2);
            \draw[black, thick, latex-] (circ2) -- (circ3);
            \draw[black, thick, latex-] (circ3) -- (circ4);
            \draw[black, thick, latex-] (circ4) -- (circ5);
            \draw[black, thick, latex-] (circ5) -- (circ6);
            \draw[black, thick, latex-] (circ6) -- (circ1);
        \end{tikzpicture}
    \end{subfigure}
    \hfill
    \begin{subfigure}[b]{0.3\textwidth}
        \centering
        \begin{tikzpicture}[scale=1.5, inner sep=.75mm, minicirc/.style={circle,draw=black,fill=black,thick}]

            \node (circ1) at ( 60:1) [minicirc] {};
            \node (circ2) at (120:1) [minicirc] {};
            \node (circ3) at (180:1) [minicirc] {};
            \node (circ4) at (240:1) [minicirc] {};
            \node (circ5) at (300:1) [minicirc] {};
            \node (circ6) at (360:1) [minicirc] {}; 

            \draw[black, thick, latex-] (circ1) -- (circ3);
            \draw[black, thick, latex-] (circ3) -- (circ5);
            \draw[black, thick, latex-] (circ5) -- (circ2);
            \draw[black, thick, latex-] (circ2) -- (circ6);
            \draw[black, thick, latex-] (circ6) -- (circ4);
            \draw[black, thick, latex-] (circ4) -- (circ1);
        \end{tikzpicture}
    \end{subfigure}
    
    \caption{Eulerian Digraph (left) decomposed into two edge-disjoint, Hamiltonian cycles (center and right). Each cycle corresponds to a circuit.}
    \label{figure:hamiltonian-decomp}
\end{figure}

It follows that if we can efficiently determine the circuit distance for every pair of vertices of a circulation problem, then we can efficiently determine if the graph can be decomposed into two Hamiltonian cycles.  

\begin{cor} \label{cor:2step-nphard}
The problem \textsc{Circ-Dist} is at least as hard as determining if an Eulerian digraph with in-degree and out-degree $2$ can be decomposed into two Hamiltonian dicycles.
\end{cor}

\begin{proof} 
Let $G$ be an Eulerian digraph with in-degree and out-degree $2$ at each node. We consider the circulation problem on $G$ with lower capacity $0$ and upper capacity $1$ at each node. We let $\mathbf{0}$ be the zero flow and $\vex$ be full flow. By \autoref{thm:2stephard}, $G$ can be decomposed into two Hamiltonian dicycles
if and only if \textsc{Circ-Dist} returns true for this instance with $k=2.$
\end{proof}

The circuit walk $((\veg_1, 1), (\veg_2, 1))$ used in the proof of \autoref{thm:2stephard} is a sign-compatible circuit decomposition of $\vex - \mathbf{0}$. Therefore, \textsc{SCM-Circ-Dist} is also NP-hard.

\begin{cor} \label{cor:scm-2step-hard}
The problem \textsc{SCM-Circ-Dist} is at least as hard as determining if an Eulerian digraph with in-degree and out-degree $2$ can be decomposed into two Hamiltonian dicycles.
\end{cor}

\begin{proof} 
The circuits used in the proof of \autoref{thm:2stephard} are sign-compatible with each other and the vector $\vex - \mathbf{0}$ because they correspond to edge-disjoint cycles whose union is the support of $\vex$. We note that $\vex$ is a vector of all-ones. The two circuit steps $\veg_1$ and $\veg_2$ correspond to edge-disjoint cycles, and therefore, $\veg_1$ and $\veg_2$ are sign-compatible. Thus, this walk is sign-compatible. 
\end{proof}

Additionally, it is also NP-hard to determine the fewest number of circuits needed for a sign-compatible circuit decomposition.

\begin{cor}\label{cor:scm-decomp-hard}
    The problem \textsc{SC-Decomp-Dist} is at least as hard as determining if an Eulerian digraph with in-degree and out-degree $2$ can be decomposed into two Hamiltonian dicycles.
\end{cor}

\begin{proof}
Consider a $0/1$-circulation polytope on an Eulerian digraph with in-degree and out-degree $2$ at each node. Let $\veo$ be the zero flow and $\vex$ be the full flow. By \autoref{cor:scm-2step-hard}, there exists a sign-compatible circuit walk of length $2$ from $\veo$ to $\vex$ if and only if the graph can be decomposed into two Hamiltonian cycles.

We will show if there exists a sign-compatible circuit decomposition of $\vex - \veo$ of length $2$, then that circuit decomposition is a circuit walk. Let $((\veg_1, \lambda_1),(\veg_2, \lambda_2))$ be a sign-compatible circuit decomposition from $\veo$ to $\vex$. Suppose the step lengths are not maximal. If $\lambda_1$ is not maximal, then the circuit $\veg_2$ does not correspond to a cycle. Likewise, if $\lambda_2$ is not maximal, then $\veg_1$ does not correspond to a cycle.
\end{proof}

We have established that the problems \textsc{Circ-Dist} and \textsc{SCM-Circ-Dist} are NP-hard. Next, we will show that if the value of $k$ is polynomial in the size of the input then \textsc{Circ-Dist} and \textsc{SCM-Circ-Dist} are NP-complete.

\subsection{Completeness of Circuit Distance}\label{subsec:cd_complete}

An instance of \textsc{Circ-Dist} is given by the tuple $(P, \vex, \vey, k)$ where $P$ is a polyhedron in general-form, $\vex, \vey \in P$ are vertices, and $k \in \mathbb{N}.$ A certificate for this problem is a circuit walk $W=((\veg_1, \lambda_1), \ldots, (\veg_r, \lambda_r))$, where each $\veg_i$ is a circuit, $\lambda_i$ is the corresponding step length, and $r \leq k$. When the value of $k$ is not polynomial in size of $P$ then the length of $W$ may be exponential in the input size. Because verifying that $W$ is a certificate requires verifying that each circuit step follows a circuit direction, this requires exponentially many checks. 

In the following, we assume the value of $k$ is polynomial in size of $P$.
To prove that \textsc{Circ-Dist} is NP-complete, we show that a yes-solution can be verified in polynomial time. 
First, we give a polynomial-time algorithm to verify if $W$ is indeed a circuit walk from $\vex$ to $\vey$ with at most $k$ steps. Then, we explain, employing standard arguments, why the encoding length of the solution is polynomial in the encoding length of the problem. 

\paragraph{Polynomiality of Verification.}
For a general-form, rational polyhedron $P = \{ \vex : A\vex = \veb, \, B\vex \leq \ved \}$, with $A \in \R^{m\times n}$, and a vector $\veg \in \R^n$, we can determine if $\veg$ is a circuit in polynomial-time. Even using a naive algorithm, this can be verified in $O(mn^3)$ time. 

\begin{lem}
    Let $P = \{ \vex : A\vex=\veb, \, B\vex \leq \ved \}$ be a rational polyhedron, let $\vex, \vey \in P$ be vertices and let $W = ((\veg_1, \lambda_1), \ldots, (\veg_r, \lambda_r))$ be a circuit walk. It can be verified in polynomial time that $W$ is a circuit walk with at most $k$ steps for some $k$ whose value is polynomial in the input size of $P$.
\end{lem}

\begin{proof}
    First, we check if the length of $W$ is at most $k$, which can be done in polynomial time. To verify that a vector is a circuit we check that it is in the kernel of $A$ and satisfies support-minimality. To verify support-minimality, we let $B'$ be the row-submatrix of $B$ consisting of all rows $B_i$ such that $B^\T_i\veg = 0$. The vector $\veg$ is a circuit if the rank of  $\binom{A}{B'}$ is $n-1$. We can efficiently check if $\veg$ has coprime integer components by checking if each component is an integer and verifying the entries are coprime via Euclid's algorithm. 

    We verify $W$ is a maximal feasible circuit walk from $\vex$ to $\vey$ by checking each requirement listed in \autoref{defn:circuit_walk}. When the value of $k$ is polynomial in the size of the input, it takes polynomial time to verify that $|W| \leq k$. Verification that requirements $(2)$ and $(3)$ are satisfied is clearly polynomial. Requirement $(1)$ can be verified in polynomial time by iterating over each step and verifying that each $\veg_i$ is a circuit. Verifying membership of a rational point in a general-form polyhedron is efficient. Thus, requirement $(4)$ can be verified in polynomial time by iteratively adding each step $\lambda_i \veg_i$ to the current point $\vex_i$. Finally, requirement $(5)$ can be verified by iterating over each $\veg_i$ and computing the maximal step length. The maximal step length of a circuit direction can be computed efficiently; \autoref{lem:circ-step-size-poly} (below) fully explains this process. Sign-compatibility of the whole walk can be verified in polynomial time by verifying that each $\veg_i$ is sign-compatible with each other $\veg_j$ with respect to $B$; this requires $O(k^2)$ checks.
\end{proof} 

\paragraph{Polynomiality of the Encoding of the Solution.}
To conclude that \textsc{Circ-Dist} is NP-complete, the encoding length of the certificate must be polynomial in the encoding length of the problem. We show that the encoding length of any circuit and any step length is polynomial in the encoding length of polyhedron $P$, which itself is part of the input of \textsc{Circ-Dist}.

\begin{lem} \label{lem:circ-poly-size}
Let $P$ be a rational polyhedron in general-form. The encoding length of a circuit $\veg$ is polynomial in the encoding length of $P$.
\end{lem}

\begin{proof}
Let $P = \{ \vex : A\vex = \veb, \, B\vex \leq \ved \}$ be a rational polyhedron, and suppose that  $B$ has $k$ rows. Every circuit of $P$ can be constructed in the following manner. For $I \subseteq \{ 1, \ldots, k \}$, we define $B_I$ to be the row-submatrix of $B$ with rows indexed by $I.$ The subset $I$ defines a circuit $\veg$ if and only if there exists a nonzero 
$\veg \in \ker \binom{A}{B_I}$
and for any $I' \supsetneq I$, $\ker \binom{A}{B_{I'}}$
is trivial.

Given a set $I$ satisfying these conditions, the circuit $\veg$ can be computed by Gaussian elimination on $\binom{A}{B_I}$. The circuit $\veg$ is a nonzero element of the kernel scaled to coprime integral components. Because the kernel has dimension $1$, the circuit is unique (up to scaling by $-1$). Therefore, the encoding size of a circuit $\veg$ (before normalization) is polynomial in the encoding size of the problem.  It remains to exhibit that converting to coprime integral components retains polynomiality of encoding size. Note that an entry of $\veg$ may be rational but not integral. Scaling $\veg$ by the product of the denominators retains polynomiality of encoding size and returns an integral vector. To convert to a coprime integer representation, we compute the greatest common divisor (GCD) of all entries of $\veg$ and divide $\veg$ by the GCD. This operation also retains polynomiality of encoding size and returns the circuit with coprime integer components.
\end{proof}

Now that we have verified that the size of a circuit $\veg$ is polynomial in the encoding length of $P$, we will verify that the encoding length of each step length $\lambda_i$ is polynomial in the encoding length of $P$ and that the resulting point is polynomial in the encoding lengths of $\vex$ and $P$.\\

\begin{lem} \label{lem:circ-step-size-poly}
Let $P$ be a polyhedron in general-form. The encoding lengths of the maximal step length $\lambda$ of a circuit $\veg$ from $\vex \in P$ and of the resulting $\vex_1=\vex+\lambda \veg \in P$ are polynomial in the encoding lengths of $\vex$ and $P$.
\end{lem}

\begin{proof}
Let $P = \{ \vex : A\vex = \veb, \, B\vex \leq \ved \}$, $\vex \in P$, and $\veg \in \mathcal{C}(P)$ such that $\veg$ is feasible at $\vex$. The maximal step length of $\veg$ can be computed via a minimum ratio test. That is, 
\[
    \lambda = \min \left\{ \frac{\ved(i) - B^\T_i\vex}{B^\T_i\veg} : B^\T_i\veg > 0 \right\}.
\] 
The value of $\lambda$ can be found using arithmetic operations on the input and $\veg$, which by \autoref{lem:circ-poly-size} is polynomial in the size of $P$. It follows that the encoding length of $\lambda$ is polynomial in the encoding length of the problem. Because $\vex$, $\veg$, and $\lambda$ are of polynomial encoding size, so is $\vex + \lambda \veg$.
\end{proof}

By repeated application of \autoref{lem:circ-step-size-poly}, we can verify that all maximum step lengths have polynomial encoding size. Thus, \textsc{Circ-Dist}, \textsc{SC-Decomp-Dist}, and \textsc{SCM-Circ-Dist} are NP-complete when the value of $k$ is polynomial in the encoding size. 

\subsection{Hardness of Geometric Circuit Distance}\label{subsec:min_length_hard}

Recall that our second notion of a short circuit is a circuit walk with small total length. Given a norm $\lVert \cdot \rVert$, we call this problem $\lVert \cdot \rVert$\textsc{-Dist}. 
\begin{prob}[$\lVert \cdot \rVert$\textsc{-Dist}]\label{prob:norm-dist}
    Let $P = \{ \vex : A\vex = \veb, \, B\vex \leq \ved \}$ be a rational polyhedron, let $\vex, \vey \in P$ be vertices, let $\lVert \cdot \rVert$ be a norm on $P$, and let $d \in \R$. Determine if there exists circuit walk $((\veg_1, \lambda_1), \ldots, (\veg_r, \lambda_r))$ from $\vex$ to $\vey$ such that $\sum_{i=1}^r \lVert \lambda_i \veg_i \rVert \leq d$.
\end{prob}

We will show that $\lVert \cdot \rVert$\textsc{-Dist} is NP-complete for the $p$-norms $\lVert \cdot \rVert_p$ with $1 < p \leq \infty$. In \autoref{lem:p-norm}, we consider $p$-norms with finite $p$, and in \autoref{lem:sup-norm}, we consider the sup-norm. Throughout, we consider a $0/1$-circulation polytope on an Eulerian digraph with in-degree and out-degree $2$.

As an aside, we remark that our proof techniques do not immediately transfer to a proof that $\lVert \cdot \rVert_1$-\textsc{Dist} is NP-hard. In fact, for network-flow polyhedra, we can show that if a sign-compatible circuit walk exists, then it is the minimizer of $\lVert \cdot \rVert_1$-\textsc{Dist}. To see this, note that for any circuit walk $(\veg_1, \ldots, \veg_r)$, the $1$-norm distance sums the ``total change'' of each constraint $0 \leq \vex(i) \leq \veu(i)$ over each step of the walk. A sign-compatible circuit walk, if it exists, minimizes this total change.

\begin{lem}\label{lem:p-norm}
    Let $G$ be an Eulerian digraph where each node has in-degree and out-degree $2$. Consider the $0/1$-circulation problem on $G$. Let $\mathbf{0}$ denote the zero flow and $\vex$ denote the full flow. 
    Let $1 < p < \infty$. The solution to $\lVert \cdot \rVert_p$\textsc{-Dist} with input $(P, \mathbf{0}, \vex)$ is $2\sqrt[p]{n}$ if and only if $G$ can be decomposed into two Hamiltonian cycles. 
\end{lem}

\begin{proof}
    Recall that for a $0/1$-circulation polytope all maximal circuit steps have step length $1$ and each circuit satisfies $\veg \in \{ 0, \pm 1 \}^{2n}$. Suppose  that $G$ can be decomposed into two Hamiltonian cycles. The length of the circuit walk corresponding to the two Hamiltonian cycles with respect to $\lVert \cdot \rVert_p$ is $2\sqrt[p]{n}$. 
    
    Now suppose that $G$ cannot be decomposed into two Hamiltonian cycles. Let $((\veg_1, 1), \ldots, (\veg_r, 1))$ be any circuit walk from $\mathbf{0}$ to $\vex$ with $r \geq 3$. For each $1 \leq i \leq r$, we have $2 \leq \lVert \veg_i \rVert_1 \leq n$ and $\sum_i \lVert \veg_i \rVert_1 \geq 2n$. Interpreting $z_i$ as the number of edges in a cycle, the value $\sum_i \lVert \veg_i \rVert_p$ is bounded below by the solution to the following constrained optimization problem:
    \begin{align*}
        \text{min} \quad & z_1^{1/p} + \cdots + z_{r}^{1/p}  \\
        \text{s.t.} \quad & 2 \leq z_i \leq n \text{ for each } i \in \{1, \ldots, r-1 \} \\
        & \sum_{i=1}^{r} z_i \geq 2n.
    \end{align*}
    In an optimal solution for $r \geq 3$, at least one cycle has fewer than $n$ edges---otherwise the graph can be decomposed into two Hamiltonian cycles. Therefore, for some $i$, $z_i < n$.
    The problem essentially finds the minimum $\lVert \cdot \rVert_p$-length of a circuit walk consisting of $r$ circuits and is relaxed to continuous variables.
    Suppose that $(z_1, \ldots, z_{r})$ is a feasible solution. We claim  
    \[
    z_1^{1/p} + \cdots + z_{r}^{1/p} > 2\sqrt[p]{n}.
    \]
    To see this, consider the following
    \begin{align*}
        \frac{1}{2(n^{1/p})}\left( z_1^{1/p} + \cdots + z_{r}^{1/p} \right) =\, & \frac{1}{2} \left( \left( \frac{z_1}{n} \right)^{1/p} + \cdots + \left( \frac{z_{r}}{n} \right)^{1/p} \right) \\
        >\, & \frac{1}{2} \left(  \frac{z_1}{n}  + \cdots + \frac{z_{r}}{n} \right) \geq \frac{1}{2} \left( \frac{2n}{n} \right) = 1.
    \end{align*}
    The strict inequality follows because for some $i$, $0 < z_i/n < 1$ and for all $i$, $0 < z_i/n \leq 1$.
    If $G$ cannot be decomposed into two Hamiltonian cycles, then any circuit walk from $\mathbf{0}$ to $\vex$ must contain at least $3$ steps. Therefore, the shortest $\lVert \cdot \rVert_p$ length is greater than $2\sqrt[p]{n}$.
\end{proof}

We prove an analogous statement for the sup-norm.

\begin{lem}\label{lem:sup-norm}
    Let $G$ be an Eulerian digraph where each node has in-degree and out-degree $2$. Consider the $0/1$-circulation problem on $G$. Let $\mathbf{0}$ denote the zero flow and $\vex$ denote the full flow. 
    The solution to $\lVert \cdot \rVert_{\infty}$\textsc{-Dist} with input $(P, \mathbf{0}, \vex)$ is $2$ if and only if $G$ can be decomposed into two Hamiltonian cycles. 
\end{lem}

\begin{proof}
    Because each circuit $\veg$ has step length $1$ and $\veg \in \{ 0, \pm 1 \}^{2n}$, $\lVert \veg \rVert_{\infty} = 1$. For a circuit walk $(\veg_1, \ldots, \veg_k)$, we have
    \[
        \sum_{i=1}^k \lVert \veg_i \rVert_{\infty} = k.
    \]
    By \autoref{thm:2stephard}, the minimum of $\sum_{i=1}^k \lVert \veg_i \rVert_{\infty}$ over all circuit walks is $2$ if and only if $G$ can be decomposed into two Hamiltonian cycles.
\end{proof}

Lemmas \ref{lem:p-norm} and \ref{lem:sup-norm} show that in the setting of a $0/1$-circulation polytope on an Eulerian digraph, the minimizer of the geometric distance is a sign-compatible circuit walk. Therefore, the related problems of determining the minimal geometric distance of a sign-compatible circuit walk and sign-compatible circuit decomposition are NP-hard. 

\subsection{Additional Remarks}\label{subsec:addtl_remarks}

Determining whether these problems can be efficiently approximated remains interesting. Given a $0/1$-network flow polytope on an Eulerian digraph $G$, the circuit distance between the zero-flow and full-flow is bounded above by the number of cycles in a minimum edge-disjoint cycle decomposition of $G$; the sign-compatible circuit distance between the zero-flow and full-flow is exactly this value. Further, the length of a sign-compatible circuit walk between the zero-flow and full-flow gives a lower bound on the length of the longest cycle in the graph. That is, if the graph has $n$ edges and there exists a sign-compatible circuit walk with $k$ steps, then the underlying graph has a cycle of length at least $n/k$.

We believe that this observation can lead to a proof that it is hard to efficiently approximate \textsc{Circ-Dist} and \textsc{SCM-Circ-Dist}. The authors of \cite{bhk-04} proved that for every $\e > 0$, there is no deterministic, polynomial-time approximation algorithm for the longest cycle in a Hamiltonian digraph with bounded out-degree on $n$ vertices with performance ratio $n^{1-\e}$, unless P = NP. In contrast, for many graphs with bounded outdegree (such as expander graphs), long cycles can be found efficiently \cite{bhk-04}. To transfer the hardness, one would need to construct a class of general digraphs that contain a Hamiltonian cycle and that can be decomposed into a small number of cycles, but for which there is no deterministic, polynomial-time approximation algorithm for the longest cycle.

\section{Stepping into Particular Facets is Hard }\label{sec:single_step_hard}

In the previous section, we established that \textsc{Circ-Dist}, \textsc{SC-Decomp-Dist}, and \textsc{SCM-Circ-Dist} are NP-hard problems, which follows from reduction to the Hamiltonian cycle problem. On the other hand, for the circulation problem on an Eulerian digraph with unit capacities, a sign-compatible circuit walk between the zero-flow $\veo$ and the full flow $\vex$ can be constructed efficiently. To do so, one follows the standard process for computing a sign-compatible \textit{decomposition}, which in this case is guaranteed to also be a maximal walk.  In particular one computes a cycle in the underlying digraph; this yields a circuit. Then, one deletes this cycle from the graph and finds a new cycle to remove. This process can be iterated until no edges remain. Because the graph is Eulerian, this process is guaranteed to terminate.  However, this walk is not guaranteed to be of the minimum possible length. While a \emph{shortest} sign-compatible circuit walk is hard to compute, a sign-compatible circuit walk is not.

sign-compatible circuit walks are short because each step of the walk enters a new facet incident to the target, and thus, the number of steps is bounded by the dimension. Walks with this property do not always exist, and in this section we prove that it is an NP-complete problem to determine if there exists a circuit that enters a facet incident to the target vertex. For a given polyhedron with vertex $\vev$, we define $\mathcal{F}(\vev)$ to be the union of facets incident to $\vev$. We will prove the following two problems are NP-complete.

\begin{prob}[\textsc{Facet-Step}]\label{prob:facet_step}
Given a rational polyhedron $P = \{ \vex : A\vex = \veb, \, B \vex \leq \ved \}$, a facet $F$ of $P$ defined by the tight inequality $B_i \vex = \ved(i)$, 
and a point $\veu \in P$, compute a circuit $\veg$ such that $\step(\veu, \veg) \in F$ or determine that no such $g$ exists. 
\end{prob}

\begin{prob}[\textsc{Incident-Facet-Step}]\label{prob:incident_facet_step}
Given a rational polyhedron $P = \{ \vex : A\vex = \veb, \, B \vex \leq \ved \}$ a point $
\veu$ of $P$, and a vertex $\vev$ of $P$, compute a circuit $\veg$ such that $\step(\veu, \veg) \in \mathcal{F}(\vev)$ or determine that no such $\veg$ exists. 
\end{prob}

We leave the hardness of determining if a sign-compatible circuit walk exists as an open question. One can ask whether it is even possible to compute the first step of such a walk: 

\begin{prob}[\textsc{SCM-Step}]\label{prob:scm_step}
    Given a rational polyhedron $P = \{ \vex : A\vex = \veb, B\vex \leq \ved \}$ and vertices $\vev, \vew \in P$, compute a circuit $\veg$ such that 
    \begin{enumerate}
        \item $\veg$ and $\vew - \vev$ are sign-compatible with respect to $B$, 
        \item $\step (\vev, \veg) \neq \vev$,
        \item $\veg$ and $\vew - \step(\vev, \veg)$ are sign-compatible with respect to $B$,  
    \end{enumerate}
    or determine that no such $\veg$ exists.
\end{prob}

We call such a circuit $\veg$ a \textit{sign-compatible maximal circuit} (SCM circuit) between $\vev$ and $\vew$, and we call $\step(\vev, \veg)$ a \textit{sign-compatible maximal step} (SCM step) from $\vev$ to $\vew$. The purpose of the third condition is to ensure that $\veg$ can plausibly be the first step of a sign-compatible walk from $\vev$ to $\vew$.  In particular, if $\veg$ and $\vew - \step(\vev,\veg)$ are not sign-compatible, then any circuit walk from $\vev$ to $\vew$ whose first step is $\veg$ cannot be a sign-compatible walk.

\textsc{SCM-Step} is not equivalent to either \textsc{Facet-Step} or \textsc{Incident-Facet-Step}: there exist polyhedra that are yes-instances for \textsc{Incident-Facet-Step} and no-instances for \textsc{SCM-Step} (see \autoref{figure:incident_not_scm}). However, the hardness of \textsc{Facet-Step} and \textsc{Incident-Facet-Step} suggest that an SCM step (and therefore, an entire SCM circuit walk) is hard to compute.

\begin{figure}
    \centering
    \begin{tikzpicture}[scale=.7, inner sep=.75mm, minicirc/.style={circle,draw=black,fill=black,thick}]
        \draw[very thick] (0,0) -- (1,2) -- (9,3) -- (10,0) -- (9,-3) -- (1,-2) -- cycle;

        \draw[blue, fill=blue!15] (0,0) -- (5, 5/8) -- (10,0) -- (5, -5/8) -- cycle;

        \draw[red, very thick, -latex] (0,0) -- (48/5, 6/5);

        \node at (-.35, .2) {$\vex$};
        \node at (10.35, .2) {$\vey$};

        \node (circ1) at (0,0) [minicirc] {};
        \node (circ2) at (10,0) [minicirc] {};
    \end{tikzpicture}
    \caption{Polytope with a yes-solution to \textsc{Incident-Facet-Step} and a no-solution to \textsc{SCM-Step}. The red vector indicates a circuit step at $\vex$ which intersects a facet incident to $\vey$.}
    \label{figure:incident_not_scm}
\end{figure}
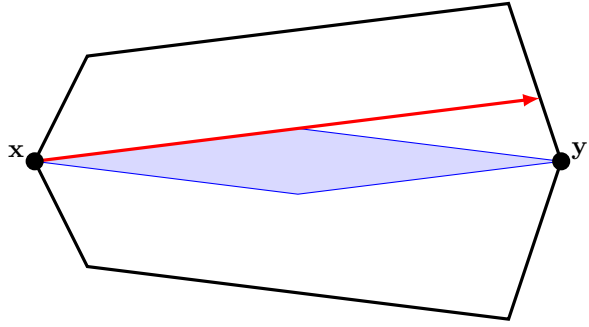

We begin by showing that \textsc{Facet-Step} is at least as hard as \textsc{Face-Step}, which is a known NP-hard problem \cite{dks-22}: 

\begin{prob}[\textsc{Face-Step}]\label{prob:face_step}    
Given a rational polyhedron $P = \{ \vex : A \vex = \veb, B\vex \leq \ved \}$, a face $G$ of $P$ defined by the tight inequalities $B_I \vex = \ved(I)$, and a point $\veu \in P$, compute a circuit $\veg$ such that there exists $\step(\veu, \veg) \in G$, or determine that no such $\veg$ exists. 
\end{prob}

To do so, we follow the strategy of De Loera, Kafer, and Sanit\`a in \cite{dks-22}.  In this work, \textsc{Face-Step} is proven to be NP-hard by showing that it is equivalent to finding a so-called \textit{greatest descent circuit} of a fractional-matching polytope on a bipartite graph.  Our strategy will modify their construction, and so we first introduce the necessary background.

\begin{defn}[Bipartite Matching Polytope]\label{defn:bipart_match}
    Let $G = (V, E)$ be a bipartite graph. The bipartite matching polytope on $G$ is as follows:
    \[
        \left\{ \vex \in \R^{|E|} : \sum_{ i : ij \in E} \vex(ij) \leq 1 \text{ for } j \in V, \, \vex(ij) \geq 0 \right\}.
    \]
\end{defn}

Given a graph $G = (V,E)$ and a matching $M$ on $G$, the \textit{characteristic vector $\vex_M$ of M} is the vector $\vex_M \in \{0, 1\}^{|E|}$ with entry $\vex_M(e) = 1$ if edge $e$ is included in the matching and $\vex_M(e) = 0$ otherwise. The matching polytope is the convex hull of the collection of all $\vex_M$.  If $G$ is not bipartite, then the so-called ``odd-set inequalities'' need to be included to obtain a description of the convex hull of all characteristic vectors of matchings of $G$. When $G$ is bipartite, the absence of odd cycles makes these constraints redundant. We note that the description in \autoref{defn:bipart_match} for the matching polytope of a bipartite graph has a constraint matrix that is totally-unimodular (TU).

\subsection{Hardness of \textsc{Facet-Step}}\label{subsec:facet-step}

We say that a polyhedron $P$ has the \emph{vertex walk property} if, for each vertex $\vev \in P$ and for each circuit $\veg$, $\step(\vev,\veg)$ is a vertex of $P$. Every $0/1$-TU polytope has the vertex circuit walk property \cite{bv-17}, and therefore, the bipartite matching polytope has this property. 

The circuits of the bipartite matching polytope are well-understood and are precisely its edge directions: a vector $\veg$ is a circuit or edge direction of the bipartite matching polytope if and only if there exists a path or cycle in $G$ with edge set $Q$ such that $\veg \in \set{\pm1,0}^E$, $\veg(e)= 0$ if and only if $e\notin Q$, and $\veg(\delta(v)) = 0$ if $\vev$ is incident to two edges of $Q$.  That is, $\veg$ alternates having values 1 and $-1$ as we traverse $Q$.

We now describe the construction that will allow us to reduce \textsc{Facet-Step} to the problem of finding a directed Hamiltonian path in a directed graph.
Let $D=(N,A)$ be a directed graph which contains two distinguished nodes $s, t \in N$ such that there exists a Hamiltonian path from $s$ to $t$. We will associate to $D$ an undirected bipartite graph $H = (V,E)$.  Ultimately, we will show that if we can efficiently solve \textsc{Facet-Step} in the bipartite matching polytope on $H$, then we can efficiently compute a Hamiltonian path from $s$ to $t$ in $D$.

We construct $H$ in the following way: 
For each $v \in N \setminus \{ t \}$, $V$ contains two copies of $v$: $v_a$, $v_b \in V$. For each $v \in N \setminus \{ t \}$, $E$ contains the edge $\{v_a, v_b \}$. For each (directed) edge $(v, u) \in A$, $E$ contains edge $\{v_b, u_a\} \in E$. $V$ also contains vertex $t$ and and $E$ contains the edges $\{u_b, t \}$ for each edge $(u, t) \in A$. Finally, $D$ contains two new vertices $s', t' \in V$ and two new edges $\{s', s_a \}$, $\{t, t' \} \in E$. A visualization of this construction is presented in \autoref{figure:aux_graph}.

\begin{figure}
    \centering
    \begin{tikzpicture}[scale = 1.33]
        \draw[very thick, -latex] (0, 0) -- (0, 1);
        \draw[very thick, -latex] (0, 1) -- (1, 0);
        \draw[very thick, -latex] (0, 0) -- (1, 1);
        \draw[very thick, -latex] (1, 0) -- (0, 0);
        \draw[very thick, -latex] (1, 0) -- (1, 1);

        \draw[fill=black] (0, 0) circle (2pt);
        \draw[fill=black] (0, 1) circle (2pt);
        \draw[fill=black] (1, 0) circle (2pt);
        \draw[fill=black] (1, 1) circle (2pt);

        \node at (-.2, -.2) {$s$};
        \node at (-.2, 1.2) {$x$};
        \node at (1.2, 1.2) {$t$};
        \node at (1.2, -.2) {$y$};

        \draw[very thick] (4, 0) -- (4, 1);
        \draw[very thick] (5, 0) -- (5, 1);
        \draw[very thick] (6, 0) -- (6, 1);
        \draw[very thick] (3, 1.5) -- (4, 1);
        \draw[very thick] (4, 0) -- (5, 1);
        \draw[very thick] (4, 0) -- (7, 1);
        \draw[very thick] (5, 0) -- (6, 1);
        \draw[very thick] (6, 0) -- (7, 1);
        \draw[very thick] (7, 1) -- (8, 1.5);

        \draw[fill=black] (3, 1.5) circle (2pt);
        \draw[fill=black] (4, 1) circle (2pt);
        \draw[fill=black] (4, 0) circle (2pt);
        \draw[fill=black] (5, 1) circle (2pt);
        \draw[fill=black] (5, 0) circle (2pt);
        \draw[fill=black] (6, 1) circle (2pt);
        \draw[fill=black] (6, 0) circle (2pt);
        \draw[fill=black] (7, 1) circle (2pt);
        \draw[fill=black] (8, 1.5) circle (2pt);

        \node at (2.8, 1.7) {$s'$};
        \node at (4.2, 1.2) {$s_a$};
        \node at (5.2, 1.2) {$x_a$};
        \node at (6.2, 1.2) {$y_a$};
        \node at (7.18, 1.3) {$t$};
        \node at (8.2, 1.7) {$t'$};
        \node at (3.8, -.2) {$s_b$};
        \node at (4.8, -.2) {$x_b$};
        \node at (5.8, -.2) {$y_b$};
    \end{tikzpicture}
    \caption{Construction of the auxiliary graph $H$ (right) from a digraph $D$ (left) for the proof of \autoref{thm:facet_step_hard}.}
    \label{figure:aux_graph}
\end{figure}
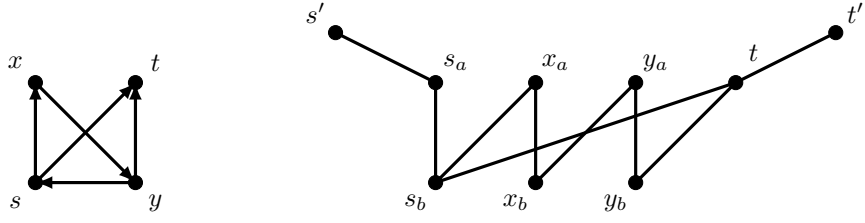

We denote by $P(H)$ the bipartite matching polytope on $H$, and we assume that $P(H)$ has a description of the form provided in \autoref{defn:bipart_match}. Let $n = |N|$ be the number of nodes of $D$, and let $W \gg 2n$. We define the following objective function, $\vecc$, on the edges of $H$: 
\[
    \vecc(e) = \begin{cases}
        \,\,\,\,\,0, & \text{if } e = v_av_b \text{ for } v \in N \setminus \{ t \}; \\
        -W, & \text{if } e = s_as'; \\
        \,\,\,\,\,W, & \text{if } e = tt'; \\
        -1, & \text{otherwise.}
    \end{cases}
\]
Because $D$ contains a Hamiltonian path from $s$ to $t$, it is easy to check that $\omo \coloneqq -W - n + 1 = \min \{ \vecc^\T \vex : \vex \in P(H) \}$ and that if $M$ is any matching in $H$ whose characteristic vector $\vex_M$ satisfies $\vecc^\T\vex_M = \omo$, then $\vex_M$ is a perfect matching in $H$. We are ready to prove that \textsc{Facet-Step} is NP-complete.

\begin{thm} \label{thm:facet_step_hard}
    \textsc{Facet-Step} is NP-complete.
\end{thm}

\begin{proof}
    We first establish that \textsc{Facet-Step} is NP-hard.
    Let $\vev \in P(H)$ be the vertex corresponding to the matching $\{ v_a, v_b\}$ for each $v \in N\setminus \{t \}$ and $\{t, t' \}$.  Let $G = \set{\vex \in P(H) : \vecc^\T \vex = \omo}$.  Note that because $\omo$ is the minimum value of $\vecc^\T\vex$ over all $\vex \in P(H)$, $G$ is a face of $P(H)$. Given $\e>0$, let $\ome \coloneqq \omo+\e$. We choose $0 < \e < 1$, such that the following are satisfied:
    \begin{enumerate}[1)]
        \item $\set{\vex \in \R^n : \vecc^\T \vex = \ome}$ is a facet of the polytope $P(H,\e) \coloneqq P(H) \cap \{ \vex : \vecc^\T \vex \geq \ome \}$,
        \item if $\vew$ is a vertex of $P(H)$ and $\vecc^\T \vew > \omo$, then $\vecc^\T \vew > \ome$,
        \item the bit encoding length of $\e$ is polynomial in $n$ and $m$.
    \end{enumerate}

    In particular, a choice of $\e = \frac{1}{2}$ suffices.  Condition $1$) is satisfied because the associated hyperplane intersects the interior. Condition 2) is satisfied because $\vecc$ is an integer vector and $\vew$ is an integer vector for all vertices $\vew$ of $P(H)$, so if $\vecc^\T\vew > \omo$, then $\vecc^\T\vew \geq \omo + 1$.
    
    By construction, $\vev \in P(H,\e)$ because $\vev \in P(H)$ and $\vecc^\T \vev = W > \ome$. 
    Further, $\vev$ is a vertex of $P(H,\e)$ because it is a vertex of $P(H)$.
    Let $F$ be the facet of $P(H,\e)$ such that all $\vex \in F$ satisfy $\vecc^\T \vex = \ome$. By \cite{dks-22}, it is NP-hard to construct a circuit $\veg \in \mathcal{C}(P(H))$ such that $\step(\vev, \veg; P(H)) \in G$. We claim that if $\veg \in \mathcal{C}(P(H,\e))$ and $\step(\vev, \veg; P(H,\e)) \in F$, then $\veg \in \mathcal{C}(P(H))$ and $\step(\vev, \veg; P(H)) \in G$.

    We partition the circuits of $P(H,\e)$ into two categories: circuits $\veg$ with $\vecc^\T\veg = 0$ and circuits $\veg$ with $\vecc^\T \veg \neq 0$. If $\veg \in \mathcal{C}(P(H,\e))$ and $\vecc^\T \veg = 0$, then $\step(\vev, \veg; P(H,\e)) \notin F$. It follows that if there exists $\veg$ such that $\step(\vev, \veg; P(H,\e)) \in F$, then $\vecc^\T \veg \neq 0$. This implies that the facet $F$ does not correspond to one of the intersecting hyperplanes defining $\veg$, and therefore $\veg$ is a circuit of $P(H)$. 

    We claim that if $\step(\vev, \veg; P(H,\e)) \in F$, then $\step(\vev, \veg; P(H))\in G$. Recall that all circuit walks in $P(H)$ are vertex circuit walks. Let $\alpha > 0$ such that $\vev+\alpha \veg = \step(\vev, \veg; P(H,\e))$. While $\vev+\alpha \veg \in P(H)$, $\vev+\alpha \veg$ is not a vertex of $P(H)$ because the facet $F$ does not contain any vertices of $P(H)$. Therefore, there exists $\alpha' > \alpha$ such that $\vew := \vev + \alpha' \veg = \step(\vev, \veg; P(H))$ is a vertex of $P(H)$. We also have $\vecc^\T(\vev+ \alpha' \veg) < \vecc^\T(\vev+ \alpha \veg) =  \ome$. However, because $\vew$ is a vertex of $P(H)$ with  $\vecc^\T \vew < \ome$, by construction we have that $\vecc^\T \vew =\omo$. Therefore $\vew \in G$.

    Thus, if we can construct a circuit $\veg$ of $P(H,\e)$ such that $\step(\vev, \veg; P(H,\e)) \in F$, then we can  construct a circuit $\veg$ of $P(H)$ such that $\step(\vev, \veg; P(H)) \in G$. Therefore, \textsc{Facet-Step} is NP-hard. It remains to prove that \textsc{Facet-Step} is NP-complete.

    Given an instance to \textsc{Facet-Step} such that a circuit can be constructed, we can efficiently verify whether a circuit $\veg$ is a certificate or not.
    To do so, we compute $\vew = \step(\vev, \veg)$ by solving the linear program $\max\set{\alpha: \vev+\alpha\veg \in P}$.  Let $A_i\vex = \veb(i)$ be the tight inequality defining the facet $F$.  Then we can easily check whether $A_i\vew = \veb(i)$.  Therefore, \textsc{Facet-Step} is NP-complete.
\end{proof}

\subsection{Hardness of Incident-Facet-Step}\label{subsec:incident-facet}

We will show in \autoref{thm:incident_facet_step} that \textsc{Incident-Facet-Step} is NP-complete by reduction to \textsc{Facet-Step}. In particular, we will further modify the construction in the proof of \autoref{thm:facet_step_hard}.  Note that it does not immediately follow from the hardness of \textsc{Facet-Step} that \textsc{Incident-Facet-Step} is hard:  The ability to compute a circuit step from a point $\veu$ to a point in \textit{some} facet containing $\vev$ does not guarantee the ability to compute a circuit step to a \textit{particular} facet containing $\vev$.  

\begin{thm}\label{thm:incident_facet_step}
    \textsc{Incident-Facet-Step} is NP-complete.
\end{thm}

\begin{proof}
We first show NP-hardness.
Consider the graph $H=(V,E)$ and the bipartite matching polytope $P(H)$ as in the proof of \autoref{thm:facet_step_hard}, and assume that $H$ contains a matching whose incidence vector $\vey$ satisfies $\vecc^\T \vey =  \omo$.  Note that such a matching can be computed in polynomial time.

Let $P(H,\e)$ denote the modification of $P(H)$ as used in the proof of \autoref{thm:facet_step_hard}.  Our construction requires a point $\vez$ such that
\begin{enumerate}[a)]
    \item $\vecc^\T\vez < \ome$, but very close to $\ome$, and
    \item $\vez \in \intr(P(H) \setminus P(H,\e))$, where $\intr(\cdot)$ denotes the interior of a set.
\end{enumerate}
We will construct $\vez$ via an appropriate perturbation of $\vey$.  To that end, let $r := |E| - n + 1$.  Note that $r$ is the number of edges $e$ of $H$ such that $\vecc(e) = -1$ and $\vey(e) = 0$.  Let $\veh$ be defined by 
\[
    \veh(e) = \begin{cases}
        -1, & \text{if } \vey(e) = 1;\\
        \,\,\,\,\frac{1}{2r}, & \text{if } \vey(e) = 0. 
    \end{cases}
\]

The vector $\veh$ is a feasible direction at $\vey$ in $P(H)$, and for any inequality of $P(H)$ tight at $\vey$, that inequality is not tight at $\vey+\alpha \veh$ for any $\alpha > 0$.
We define $\Delta$ as follows
\[
    \Delta \coloneqq \vecc^\T \veh = W+n-1 -\frac{W}{2r} + \frac{r}{2r} = -\omo -\frac{W}{2r} + \frac{1}{2}.
\]
We set $T = 2r\Delta+1$. Let $\vez = \vey + \alpha \veh$ where $\alpha$ is chosen such that $\vecc^\T \vez = \omo +\frac{T-1}{T}\e$.
That is, $\alpha = \frac{(T-1)\e}{T\Delta}$. We claim that no inequality of $P(H)$ tight at $\vez$. Because $\alpha<1$, $\vez(e) > 0$ for all $e$, no non-negativity constraint of $P(H)$ is tight at $\vez$.  
Consider any constraint of the form $\sum_{i:ij\in E}\vex(ij)\leq 1$ for some $j\in V$.  
As observed above, if this constraint is tight at $\vey$, it is not tight at $\vez$. Further, all constraints of this form are tight because, as noted earlier, $\vecc^\T\vey = \omo$ implies that $\vey$ is the characteristic vector of a perfect matching. Therefore, $\vez\in \intr(P(H))$, and because $\vecc^\T \vez < \ome$, we have $\vez \in \intr(P(H)\setminus P(H,\e))$.
We have introduced a number of constants, and we now provide the reader the table below for later reference:

\begin{center}
\begin{tabular}{ c c c c }
  $W \gg 2n$ & $\omo = -W - n + 1$ & $\ome = -W -n + 1 +\e $& \, \\\\ 
  $r = |E| - n + 1$ & 
  $\Delta = \vecc^\T\veh = -\omo - \frac{W}{2r} + \frac{1}{2}$ & $T = 2r\Delta+1$ & $\alpha = \frac{(T-1)\e}{T\Delta}$ 
\end{tabular}
\end{center}

Next, we will modify $P(H,\e)$ to construct a new polyhedron, $P'(H,\e)$, which will allow us to reduce \textsc{Incident-Facet-Step} to \textsc{Facet-Step}.  We again let $F = \set{ \vex \in P(H,\e) : \vecc^\T \vex = \ome}$.  Recall that $F$ is the facet of $P(H,\e)$ that we use in the proof of~\autoref{thm:facet_step_hard} to show that \textsc{Facet-Step} is NP-hard.  Roughly, the main idea of the following reduction is to carefully replace the facet $F$ with a collection of facets all containing $\vez$ as a vertex.

Let $F_i$ denote the facet of $P(H)$ corresponding to row $i$ of the constraint matrix $B$ of $P(H)$. Without loss of generality (by relabeling, if necessary), let $B_1,\ldots, B_k$ be the rows of $B$ such that the facets $F_1 \cap P(H, \e),\ldots, F_k \cap P(H, \e)$ have non-empty intersection with $F$.
Note that each such $B_i$ can be found by solving a linear program. For $1\leq i\leq k$, let $G_i$ denote $F_i\cap F$. 
Let $\vex_i \in G_i$ and set
\begin{align*}
p(i) = \frac{B_i^\T \vex_i - B_i^\T \vez}    {\vecc^\T \vex_i - \vecc^\T \vez} .
\end{align*}
The value of $p(i)$ is independent of the particular element $\vex_i$: for all $\vex\in G_i$, we have $\vecc^\T \vex = \vecc^\T \vex_i$ and $B_i^\T \vex = B_i^\T \vex_i$.
Let $\veq_i = B_i - p(i)\vecc$.  We have that 
\begin{align*}
\veq_i^\T \vez = B_i^\T \vez - \frac{B_i^\T \vex_i - B_i^\T \vez}{\vecc^\T \vex_i - \vecc^\T \vez} \vecc^\T \vez
&= \frac{(\vecc^\T \vex_i - \vecc^\T \vez)B_i^\T \vez - (B_i^\T \vex_i - B_i^\T \vez)\vecc^\T \vez}   {\vecc^\T\vex_i - \vecc^\T\vez} \\
&= \frac{\vecc^\T\vex_i B_i^\T\vez - \vecc^\T\vez B_i^\T\vez - B_i^\T\vex_i \vecc^\T \vez + B_i^\T \vez \vecc^\T \vez}   {\vecc^\T\vex_i - \vecc^\T \vez} \\
&= \frac{\vecc^\T \vex_i B_i^\T \vez - B_i^\T \vex_i \vecc^\T \vez}    {\vecc^\T \vex_i - \vecc^\T \vez} \\
&= \frac{\vecc^\T \vex_i B_i^\T \vez - B_i^\T \vex_i \vecc^\T \vez + (B_i^\T \vex_i \vecc^\T \vex_i - B_i^\T \vex_i \vecc^\T \vex_i)}    {\vecc^\T \vex_i - \vecc^\T \vez} \\
&= \frac{(\vecc^\T \vex_i - \vecc^\T \vez)B_i^\T \vex_i - (B_i^\T \vex_i - B_i^\T \vez) \vecc^\T \vex_i} {\vecc^\T \vex_i - \vecc^\T \vez} \\
&= B_i^\T \vex_i - \frac{B_i^\T \vex_i - B_i^\T \vez}{\vecc^\T \vex_i - \vecc^\T \vez} \vecc^\T \vex_i = \veq_i^\T \vex_i.
\end{align*}
Let $\ver(i) = \veq_i^\T \vex_i$. Because $\veq_i$ is a conic combination of $B_i$ and $-\vecc$, which are themselves outer-normals to facets containing $G_i$, the inequality $\veq_i^\T \vex \leq \ver(i)$ is a valid inequality for $P(H,\e)$ that is satisfied with equality by the points of $G_i$.
Thus, for all $1\leq i \leq k$, the set $\set{ \vex \in \R^n : \veq_i^\T \vex = \ver(i)}$ is a supporting hyperplane for $G_i$ and contains the point $\vez$. 
Now, consider the polyhedron 
\[
P'(H,\e) \coloneqq P(H) \cap \set{\vex \in \R^n : \veq_i^\T \vex \leq \ver(i)\text{ for all }1\leq i \leq k}.
\]

We will consider the instance of \textsc{Incident-Facet-Step} of finding a circuit $\veg$ of $P'(H,\e)$ such that
$\step(\veu,\veg;P'(H,\e))\in\mathcal{F}(\vez)$ where, as in the proof of \autoref{thm:facet_step_hard}, we let $\veu \in P(H)$ be the vertex given by the matching $\set{t,t'}\cup\set{ v_a, v_b : v\in N\setminus \set{t}}$.

Recall that $\vez\in\intr(P(H))$, so the constraints of $P'(H,\e)$ tight at $\vez$ are precisely those of the form $\veq_i^\T\vex \leq \ver(i)$.
Note that $P'(H,\e)$ is obtained from $P(H,\e)$ by removing the inequality $\vecc^\T \vex\geq \ome$ and adding a set of inequalities that are valid for $P(H,\e)$. Therefore, all $\vex\in P'(H,\e)\setminus P(H,\e)$ satisfy $\vecc^\T \vex < \ome$, and for any such $\vex$, the inequalities tight at $\vex$ have form $\veq_i^\T \vex \leq \ver(i)$. That is, any such point that is also on the boundary of $P'(H,\e)$ is in $\mathcal{F}(\vez)$. Likewise, if $\vex\in P'(H,\e)$ satisfies that $\vecc^\T \vex = \ome$  and that $\vex$ is on the boundary of $P'(H,\e)$, then $\vex \in G_i$ for some $i$ and satisfies $\veq_i^\T \vex = \ver(i)$.  Thus, $\vex$ is in $\mathcal{F}(\vez)$.

We will show that all circuit steps in $P'(H,\e)$ that enter $\mathcal{F}(\vez)$ correspond to circuit steps in $P(H,\e)$ that enter $F$.

\begin{claim}\label{claim:circuit_for_both} A vector $\veg$ is a circuit of $P'(H,\e)$ such that $\step(\veu, \veg; P'(H,\e))\in\mathcal{F}(\vez)$ if and only if $\veg$ is a circuit of $P(H,\e)$ such that $\step(\veu, \veg; P(H,\e))\in F$.
\end{claim}

Suppose that $\veg\in\mathcal{C}(P'(H,\e))$ such that $\vew' \coloneqq \step(\veu, \veg; P'(H,\e)) \in \mathcal{F}(\vez)$.  First, we argue that $\veg$ is a circuit of $P(H,\e)$. 
Suppose for the sake of a contradiction that $\veg \notin \mathcal{C}(P(H,\e))$. Then the circuit $\veg$ satisfies $\veq_i^\T \veg = 0$ for some $1 \leq i \leq k$, and so  $\veq_i^\T \veu = \veq_i^\T \vew'$. That is,
\begin{align*}
 & B_i^\T \veu - p(i) \vecc^\T \veu = B_i^\T \vew' - p(i) \vecc^\T\vew' \\
\Rightarrow\,\,\, & B_i^\T(\veu - \vew') = p(i)\vecc^\T(\veu - \vew') \geq p(i)(W - \ome).
\end{align*}
The above inequality follows from the fact that $\vecc^\T\veu = W$ and, because $\vew'\in\mathcal{F}(\vez)$, $\vecc^\T\vew'\leq \ome$.  We have that 
$$
p(i) = \frac{B_i^\T \vex_i - B_i^\T\vez}{\vecc^\T\vex_i - \vecc^\T\vez} = \frac{B_i^\T(\vex_i - \vez)}{\frac{\e}{T}}.
$$
It follows that
\begin{equation} \label{eq:bound_improvement}
W - \ome \leq \frac{\e}{T}\frac{B_i^\T(\veu - \vew')}{B_i^\T(\vex_i- \vez)} \leq \frac{\e}{T}\frac{1}{B_i^\T(\vex_i- \vez)}.
\end{equation}
The last inequality follows from the fact that $\veu, \vew' \in P(H)$ and for any $\vex\in P(H)$ and any row $B_i$ of the constraint matrix of $P(H)$, $0\leq B_i^\T \vex\leq 1$.  We now bound the minimum size of $B_i^\T(\vex_i - \vez)$.
There are two cases:

\textbf{Case 1:} Suppose $B_i$ corresponds to a non-negativity constraint. That is, $B_i = -\vece_f$ for some edge $f$ of $H$, where $\vece_f$ is an elementary unit vector.  Recall that $\vez = \vey + \alpha \veh$ where $\alpha = \frac{(T-1)\e}{T\Delta}$. It follows that
\[
B_i^\T(\vex_i - \vez) = B_i^\T\vex_i - B_i^\T \left(\vey + \frac{(T-1)\e}{\Delta T} \veh \right) \geq \frac{(T-1)\e}{\Delta T}\frac{1}{2r}
\]
because the constraint corresponding to $B_i$ is tight at $\vex_i$ by definition. Therefore, $B_i^\T\vex_i = 0$.  By~(\ref{eq:bound_improvement}), we have that
\begin{equation*}
    W - \ome \leq \frac{\e}{T} \frac{2r\Delta T}{(T-1)\e} = \frac{2r\Delta}{T-1} = 1.
\end{equation*}
However, it can be readily seen that $W-\ome \gg 1$, a contradiction.

\textbf{Case 2:} Suppose $B_i$ corresponds to a degree constraint.  That is, $B_i = \sum_{u:uv\in E} \vece_{uv}$ for some $v\in V$. Then
\begin{align}
B_i^\T(\vex_i - \vez) &= B_i^\T \vex_i - B_i^\T \vey -B_i^\T\left(\frac{(T-1)\e}{\Delta T} \veh \right) \nonumber\\
&\geq 1 - 1  - \frac{(T-1)\e}{\Delta T}B_i^\T \veh \label{line:simplify} \\
&\geq - \frac{(T-1)\e}{\Delta T}(-1+\frac{r}{2r}) = \frac{(T-1)\e}{2\Delta T}\nonumber,
\end{align}
where~(\ref{line:simplify}) follows from the fact that $B_i^\T \vex_i = 1$ and $B_i^\T \vey \leq 1$.
By~(\ref{eq:bound_improvement}), we have that
\begin{equation}
    W - \ome \leq \frac{\e}{T} \frac{2\Delta T}{(T-1)\e} = \frac{2\Delta}{T-1}< 1,
\end{equation}
a contradiction.

In both cases we reach a contradiction, and therefore, $\veg \in \mathcal{C}(P(H,\e))$.
Now, $\vew'$ satisfies $\vecc^\T \vew'\leq \ome$, and therefore there exists a scalar $\alpha$ such that $\vew \coloneqq \veu + \alpha \veg$ satisfies $\vecc^\T \vew = \ome$.  Thus, $\vew$ is feasible for $P(H,\e)$, $\vew \in F$, and $\vew = \step(\veu, \veg; P(H,\e))$.

Now, suppose that $\veg\in\mathcal{C}(P(H,\e))$ such that $\vew = \step(\veu, \veg; P(H,\e))$ is in $F$. We first argue that $\veg$ is a circuit of $P'(H,\e)$. Suppose to the contrary $\veg \notin \mathcal{C}(P'(H, \e))$.
Then, we have $\vecc^\T \veg = 0$. However, $\vecc^\T \vew \neq \vecc^\T \veu$, which is a contradiction, and therefore, $\veg \in \mathcal{C}(P'(H, \e))$. Let $\vew' = \step(\veu, \veg; P'(H,\e))$, in particular, let $\vew' = \veu + \alpha'\veg$ for some $\alpha' > 0$. Because $\vew \in P(H,\e)\subseteq P'(H,\e)$, it follows that $\alpha' \geq \alpha$, and so $\vecc^\T \vew' \leq \vecc^\T \vew$.  Thus $\vew'\in \mathcal{F}(\vez)$; which proves \autoref{claim:circuit_for_both}.
\vspace{.3cm}

Therefore, if we can solve \textsc{Incident-Facet-Step} efficiently, then we can solve \textsc{Facet-Step} efficiently.
Thus, \textsc{Incident-Facet-Step} is NP-hard.
Now, given an instance of \textsc{Incident-Facet-Step} such that a circuit can be constructed, we can efficiently verify whether a circuit $\veg$ is a certificate or not. To do so, we compute $\vew = \step(\veu, \veg)$ by solving the linear program $\max\set{\alpha: \veu+\alpha \veg\in P}$. It follows that $\vew \in \mathcal{F}(\vev)$ if and only if for at least one inequality $B_i^\T \vex\leq \ved(i)$ of $P$, both $B_i^\T\vev = \ved(i)$ and $B_i^\T\vew = \ved(i)$. This can be easily checked for each inequality. Therefore, \textsc{Incident-Facet-Step} is NP-complete.
\end{proof}

We have shown that \textsc{Incident-Facet-Step} is NP-complete in general. However, the nature of our construction means that the polyhedron in our reduction no longer has many of the nice properties of the bipartite matching polytope from which we started. For example, it is no longer TU.  It would be interesting to determine for which classes of polyhedra \textsc{Incident-Facet-Step} is easy or hard. In the following section, we prove that it \textit{is} indeed easy in the case of TU polyhedra. That is, our approach to proving the hardness of \textsc{Incident-Facet-Step} \textit{required} eliminating the total unimodularity of the polytope we were considering.

\section{Efficient Cases for \textsc{Incident-Facet-Step} and \textsc{SCM-Step}}\label{sec:TU_easy}

Sections \ref{sec:short_walk_hard} and \ref{sec:single_step_hard} demonstrate that the construction of short circuit walks and related subproblems are hard. In general, in this section, we provide some favorable cases. First, we give polynomial-time algorithms for both \textsc{Incident-Facet-Step} and \textsc{SCM-Step} when restricting to TU polyhedra. Both algorithms construct an auxiliary polyhedron and then efficiently determine if a particular type of vertex exists. Next, we show there exist polynomial-time, graph-theoretic algorithms for both \textsc{Incident-Facet-Step} and \textsc{SCM-Step} when restricting to network-flow polytopes. Both algorithms are variants of cycle-finding algorithms on a modified version of the residual network. We conclude by showing that the so-called $(n, d)$-parallelotopes, a generalization of parallelotopes, always have sign-compatible circuit walks for all pairs of vertices. Therefore, they have a small circuit diameter.

\subsection{Efficient Algorithms for TU Polyhedra}\label{subsec:tu_easy}

We show that both \textsc{Incident-Facet-Step} and \textsc{SCM-Step} can be solved in polynomial time and space for TU polyhedra. Recall that for a circuit $\veg$ of a TU polyhedron, $\veg \in \{ 0, \pm 1 \}^n$ (see, e.g., \cite{o-10}). For any polyhedron $P$, the circuits of $P$ can be represented as vertices of an associated polyhedron.

\begin{prop}[Polyhedral Model of Circuits \cite{bv-19c}]\label{prop:polyhedral-model}
Let $P = \{ \vex \in \R^n : A\vex = \veb, \, B\vex \leq \ved \}$ and
\begin{equation} \label{eq:polyhedral_model}
\begin{aligned}
    P_{A, B} = \{(\vex, \vey^+, \vey^-) \in \R^{n+2m_B}:\,
    &A\vex = \mathbf{0}, \\
    &B\vex = \vey^+ - \vey^-, \\
    &\lVert \vey^+ \rVert_1 + \lVert \vey^- \rVert_1 = 1, \\
    &\vey^+,  \, \vey^- \geq \mathbf{0}\}.
\end{aligned}
\end{equation}
For nonzero $\veg \in \R^n$, $\veg \in \mathcal{C}(A, B)$ if and only if there exists $\alpha > 0$ and $\vey^+, \vey^- \in \R^{m_B}$ such that $(\alpha \veg, \vey^+, \vey^-) \in P_{A,B}$ is a vertex.
\end{prop}

In \autoref{prop:polyhedral-model}, the value of $\alpha$ simply scales $\veg$. Given a circuit $\veg \in \mathcal{C}(A, B)$, $\veg$ is scaled to have coprime integral components, whereas in $P_{A,B}$ the circuit $\veg$ is scaled to have unit $1$-norm. $P_{A, B}$ also has an $m_B$ additional vertices that do not correspond to circuits. For each $1 \leq i \leq m_B$, the vector formed by setting $\vey^+(i) = \vey^-(i) = 1/2$ and all other coordinates to zero is a vertex.
We will now describe a polynomial-time algorithm to solve \textsc{Incident-Facet-Step} for a TU polyhedron. Following the description, we will prove correctness and show that the algorithm is polynomial.

\begin{algo}[TU \textsc{Incident-Facet-Step}]\label{algo:TU_facet_step}
Let $P = \{ \vex : A\vex = \veb, \, B\vex \leq \ved \}$ and let $\vev, \vew \in P$ be distinct vertices. Let $m$ be the number of rows of $B$. Recall, $P_{A, B}$ is the polyhedral model of circuits of $P$.  
\\
\\
\noindent For each $j \in \{ 1, \ldots, m \}$ such that $B^\T_j\vev < B^\T_j\vew = \ved(j)$, do the following:
\begin{enumerate}
    \item Set $\kappa = B^\T_j\vew - B^\T_j\vev$.
    \item Set $\vey^-(j) = 0$.
    \item For each $i \in \{1, \ldots, m\}$, if $\ved(i)- B^\T_i \vev < \kappa$, then add the constraint $\vey^+(i) = 0$ to $P_{A, B}$.
    \item Compute a vertex $(\vex,\vey^+, \vey^-)$ on this face with $\vex(j) > 0$, or decide that one does not exist.
    \item If a vertex was found, STOP and return the vertex.
    \item If no such vertex was found, remove all constraints added in Steps 2 and 3.
\end{enumerate}
If no vertex is found for any $j$, then an SCM circuit step does not exist.
\end{algo}

\begin{lem} \label{lem:TU-Alg-Poly}
    \autoref{algo:TU_facet_step} runs in polynomial time and space.
\end{lem}

\begin{proof}
Because $B$ has $m$ rows, the loop has at most $m$ iterations---which is polynomial in the input size of the problem. It remains to show that each iteration of the loop runs in polynomial time. Steps \textit{1}, \textit{2}, \textit{3}, \textit{5} and \textit{6} run in polynomial time because the comparison and algebraic operations are efficient. Let $P_{A, B}'$ denote the polytope created after adding the constraints in Steps \textit{2} and \textit{3}. Step \textit{4} runs in polynomial time because a vertex can be found by solving the following linear program: $\max\{ \vex(j) : (\vex, \vey^+, \vey^-) \in P_{A, B}'\}$. If the maximum of $\vex(j) = 0$, then no such circuit exists.   

The additional space required to account for added constraints, for $\kappa$, and possibly for a found circuit is linear in the size of the encoding. Therefore, the algorithm uses polynomial space.
\end{proof}

To verify correctness, we show that the algorithm decides if there exists a circuit step that enters a facet incident to $\vew$ by iteratively checking each incident facet.

\begin{lem} \label{lem:TU-Alg-Poly_correct}
    \autoref{algo:TU_facet_step} correctly finds a solution to \textsc{Incident-Facet-Step}, if one exists. Otherwise, the algorithm correctly determines that no such solution exists.
\end{lem}

\begin{proof}
    Let $\vev, \vew \in P$ be vertices. If the constraint $B^\T_j\vex \leq \ved(j)$ defines a facet incident to $\vew$ that is not incident to $\vev$, then $B^\T_j\vev < B^\T_j\vew = \ved(j)$. The algorithm iterates over all facets that are incident to $\vew$ that are not incident to $\vev$. We show that, for each $j$, the algorithm determines if there exists a circuit direction that intersects the facet given by $B^\T_j\vex \leq \ved(j)$.

    Let $F = \{ \vex \in P : B^\T_j\vex = \ved(j) \}$ be a facet such that $\vew \in F$ and $\vev \notin F$. Let $\kappa = B^\T_j\vew - B^\T_j\vev$. Recall that because $\veg$ is a circuit of a TU polyhedron, $\veg \in \{ 0, \pm 1 \}^n$. Let $\lambda$ be the maximal step length of circuit direction $\veg$. We have $\vev+\lambda \veg \in F$ if and only if $\lambda = \kappa$ and $B^\T_j\veg = 1$. 

    Step \textit{3} of the algorithm sets $\vey^+(i) = 0$ if $\ved(i) - B_i \vev < \kappa$. Thus, the maximal step length of any found $\veg$ is at least $\kappa$. If a circuit with $\vex(i) = 1$ exists, then Step \textit{4} computes such a circuit. Because $B^\T_j\veg = 1$, the step length is guaranteed to be $\kappa$ and $\vev + \kappa \veg \in F$. If no such circuit exists, then the maximum to the linear program of Step \textit{4} is $0$ and the algorithm determines that there does not exist a circuit direction $\veg$ that intersects $F$.
\end{proof}

Lemmas \ref{lem:TU-Alg-Poly} and \ref{lem:TU-Alg-Poly_correct} together yield the following theorem.

\begin{thm}
    \textsc{Incident-Facet-Step} can be solved in polynomial time for TU polyhedra.
\end{thm}

\begin{proof}
    Consider an instance of \textsc{Incident-Facet-Step}. It is efficient to determine if the underlying polyhedron is TU. Then, by \autoref{lem:TU-Alg-Poly_correct}, \autoref{algo:TU_facet_step} correctly finds a solution to \textsc{Incident-Facet-Step} if one exists. By \autoref{lem:TU-Alg-Poly}, \autoref{algo:TU_facet_step} is efficient.
\end{proof}

\autoref{algo:TU_facet_step} can be easily modified to solve an instance of \textsc{Facet-Step} in polynomial time for a TU polyhedron. We will show that \autoref{algo:TU_facet_step} can also be modified to solve an instance of \textsc{SCM-Step} in polynomial time for a TU polyhedron. First, we note that for all vertices $\vev, \vew \in P$, all circuits that are sign-compatible with $\vew-\vev$ are contained in a face of $P_{A, B}$.

\begin{prop}\label{lem:sc-face}
    Let $P = \{ \vex : A\vex = \veb, \, B\vex \leq \ved \}$, $\vev, \vew \in P$ be vertices, and $\veg \in \mathcal{C}(P)$. If $\veg$ is sign-compatible with $\vew-\vev$ then $\veg$ corresponds to a vertex on the following face of $P_{A, B}$.
    \begin{align*}
        P_{A, B}^{SC} = \{(\vex, \vey^+, \vey^-) \in \R^{n+2m_B}: \, & A\vex = 0, \\
            & B\vex = \vey^+ - \vey^-, \\
            & \lVert \vey^+ \rVert_1 + \lVert \vey^- \rVert_1 = 1, \\
            & \vey^+(j) = 0 \text{ if } B^\T_j\vev \geq B^\T_j\vew, \\
            & \vey^-(j) = 0 \text{ if } B^\T_j\vev \leq B^\T_j\vew, \\
            & \vey^+,  \, \vey^- \geq \mathbf{0} \}.
    \end{align*}
\end{prop}

\begin{proof}
    Let $\vev, \vew \in P$ be vertices and let $\veg \in \mathcal{C}(P)$ be sign-compatible to $\vew - \vev$ with respect to $B$. Consider the vector $\vev_g = (\vex_g, \vey^+_g, \vey^-_g)$ defined as 
    \[
    \vev_g = (\veg(1), \ldots, \veg(n), \max \{ 0, B^\T_1\veg \}, \ldots, \max \{ 0, B^\T_m\veg \}, \max \{ 0, -B^\T_1\veg \}, \ldots, \max \{ 0, -B^\T_m\veg \}). 
    \]
    By sign-compatibility of $\veg$, if $B^\T_j\vev \geq B^\T_j\vew$ then $B^\T_j\veg \leq 0$. Therefore, $\vey_g^+(j) = 0$ when $B^\T_j\vev \geq B^\T_j\vew$. Likewise, if $B^\T_j\vev \leq B^\T_j\vew$, we have $\vey_g^-(j) = 0$.
    Thus, $(\lVert \vey_g^+ \rVert_1 + \lVert \vey_g^- \rVert_1)^{-1}\veg \in P_{A, B}^{SC}$.
\end{proof}

The problem \textsc{SCM-Step} can be solved in polynomial time and space for a TU polyhedron with a modification to \autoref{algo:TU_facet_step}. 

\begin{cor}
    \textsc{SCM-Step} can be solved in polynomial time for a TU polyhedron.
\end{cor}

\begin{proof}
    Let $P = \{ \vex : A\vex = \veb, \, B\vex \leq \ved \}$ and let $\vev, \vew \in P$ be vertices.
    Consider the following modifications to \autoref{algo:TU_facet_step}. First, replace $P_{A,B}$ with $P_{A,B}^{SC}$. Therefore, any found circuit $\veg$ is sign-compatible with $\vew-\vev$. Second, replace Step \textit{3} of \autoref{algo:TU_facet_step} with the following step
    {\it
    \begin{itemize}
        \item For each $i \in \{ 1, \ldots, m \}$, if $|B^\T_i\vew - B^\T_i\vev| < \kappa$, set $\vey^-(i) = \vey^+(i) = 0$.
    \end{itemize}}
    This new step ensures that the step length in a sign-compatible circuit decomposition will be at least $\kappa$. Therefore, the updated algorithm will find a sign-compatible maximal step, if one exists.
\end{proof}

We have shown that both \textsc{Incident-Facet-Step} and \textsc{SCM-Step} can be solved efficiently for TU polyhedra. Now, we study a special case of TU polyhedra: network-flow polyhedra.

\paragraph{Network-flow Polyhedra.}\label{par:nf_easy} In applications, network-flow polyhedra are the most common instance of TU polyhedra. For network-flow polyhedra, we show that the problems \textsc{Incident-Facet-Step}, \textsc{Facet-Step}, \textsc{Face-Step}, and \textsc{SCM-Step} all have polynomial-time graph-theoretic algorithms. Such a connection exists because the circuits of the polyhedron have a graph-theoretical interpretation.

Recall the notion of the residual network of a graph (see Section \ref{sec:short_walk_hard}). Given a network-flow polytope $P$ given by the instance $(N, A, \veb, \veu)$, the facets of $P$ are given by the constraints $\mathbf{0} \leq \vex \leq \veu$. An edge $e$ is called \emph{facet-defining} for a vertex $\vex$ if $\vex(e) \in \{ 0, \veu(e) \}$. We consider the following algorithm and claim that it solves \textsc{Incident-Facet-Step} (in polynomial time).

\begin{algo}[NF \textsc{Incident-Facet-Step}]\label{algo:NF_facet_incident}
    Let $(N, A, \veb, \veu)$ be an instance of the network-flow problem and let $G=(N,A)$. We let $\vex, \vey$ be feasible flows on $G$. Let $G(\vex)$ be the residual network with respect to $\vex$.
    Let $M_1 = \{ e \in A : \vex(e) > \vey(e) = 0 \}$ and $M_2 = \{ e \in A : \vex(e) < \vey(e) = \veu(e) \}$. Let $M = M_1 \cup M_2$. 
    \\
    \\
    \noindent For each $e \in M$, do the following:
    \begin{enumerate}
    \item Delete all edges of $G(\vex)$ with capacity less than $|\vey(e) - \vex(e)|$.
    \item Find a cycle containing $e$ or determine no such cycle exists.
    \item If a cycle is found, STOP and return the cycle. 
    \item If no cycle is found, add all edges deleted in Step $1$ back to $G(\vex)$. 
\end{enumerate}
If no cycle is found for any $e \in M$, then there does not exist a circuit $\veg$ such that $\step(\vex,\veg)\in\mathcal{F}(\vey)$.
\end{algo}

Now, we show correctness and efficiency of \autoref{algo:NF_facet_incident}.       

\begin{lem}
    \autoref{algo:NF_facet_incident} correctly computes  a solution to \textsc{Incident-Facet-Step} and runs in polynomial time. 
\end{lem}

\begin{proof}
    The circuits of a network-flow polyhedron are the simple cycles in the underlying undirected graph. The feasible circuits at vertex $\vex$ are the simple cycles in the residual network with respect to flow $\vex$.
    Therefore, if \autoref{algo:NF_facet_incident} returns a cycle, that cycle necessarily corresponds to a feasible circuit because it is a simple cycle in the residual network. The set $M_1$ contains all reversed edges that appear in the residual network, such that the original edge is facet defining for $\vey$. The set $M_2$ contains all original edges that have maximum capacity for $\vey$, and thus, are facet-defining. Therefore, the algorithm iterates over all facet-defining edges for $\vey$ that are not facet-defining for $\vex$.

    For each edge $e \in M$, if $e \in M_1$, then the capacity of $e$ in the residual network is $\vex(e) = |0 - \vex(e)| = |\vey(e) - \vex(e)|$. If $e \in M_2$, then the capacity of $e$ in the residual network is $\veu(e) - \vex(e) = |\vey(e) - \vex(e)|$. Therefore, Step \textit{1} ensures that if a circuit $\veg$ is returned, the length of a maximal step in the direction $\veg$ from $\vex$ is $|\vey(e) - \vex(e)|$ and that the resulting step is contained in the facet containing $\vey$ which is defined by $e$. If no cycle is found for any $e \in M$, then there does not exist a circuit $\veg$ such that $\step(\vex,\veg)\in\mathcal{F}(\vey)$

    The initial work and Steps \textit{1}, \textit{3}, and \textit{4} run in polynomial time. For Step \textit{2}, a cycle containing edge $e$ can be found in polynomial time using a depth-first search. Therefore, the entire algorithm runs in polynomial time. Polynomial space follows because only the graph, the residual network, and the computed cycle must be stored.
\end{proof}

Similar to the general TU case, \autoref{algo:NF_facet_incident} can be modified to solve \textsc{Facet-Step} and \textsc{SCM-Step} for network-flow polyhedra in polynomial time.

\begin{cor}
    \textsc{SCM-Step} can be solved in polynomial time for a network-flow polyhedron.
\end{cor}

\begin{proof}
    We modify \autoref{algo:NF_facet_incident} by replacing Step \textit{1} with the following steps
    {\it
    \begin{enumerate}
        \item[1a.] Set $\kappa = |\vey(e) - \vex(e)|$.
        \item[1b.] For each edge $(v_1, v_2) \in A$, if $\vex((v_1, v_2)) \leq \vey((v_1, v_2))$, delete edge $(v_2, v_1)$ in $G(\vex)$; if $\vex((v_1, v_2)) \geq \vey((v_1, v_2))$, delete edge $(v_1, v_2)$ in $G(\vex)$.
        \item[1c.] For each edge $(v_1, v_2) \in E$, if $|\vey((v_1, v_2)) - \vex((v_1, v_2))| < \kappa$, delete edges $(v_1, v_2)$ and $(v_2, v_1)$ in $G(\vex)$.
    \end{enumerate}}
    Step \textit{1b} ensures that the found cycle can appear as a circuit direction in some sign-compatible circuit decomposition. 
    Step \textit{1c} ensures that a found cycle has step length $\kappa$ and is thus a maximal step.
\end{proof}

Therefore, despite \textsc{Incident-Facet-Step} being NP-complete in general (Subsection \ref{subsec:incident-facet}), and \textsc{SCM-Step} being open in general (Section \ref{sec:conclusion}), both problems are solvable in polynomial time for TU polyhedra. In fact, both problems reduce to cycle-finding for the network-flow case and can be solved efficiently through graph-theoretic algorithms.

\subsection{Polytopes with Sign-Compatible Circuit Walks}\label{sec:alwayssc}

In Section \ref{sec:short_walk_hard}, we proved that it is NP-hard to find a shortest sign-compatible circuit walk. However, we can efficiently construct \emph{some} sign-compatible circuit walk between the vertices of the circulation polytope we considered. We here present a class of polytopes that always have a sign-compatible circuit walk between \emph{any} given pair of vertices: the so-called $(n,d)$-parallelotopes. An $n$-parallelotope is an $n$-dimensional polytope formed from the Minkowski sum of $n$ linearly independent line segments, and thus, it is a zonotope. The class of $(n,d)$-parallelotopes (\autoref{defn:nd-parallelotope}) generalizes the class of parallelotopes. We briefly provide some necessary background on $(n,d)$-parallelotopes; for a detailed treatment, see Section 6 of \cite{bv-17}, where these objects were introduced.

\begin{defn}\label{defn:nd-parallelotope}
    Let $d \leq n \in \N$ and let $P$ be a polytope. We say $P$ is an $(n,d)$-parallelotope if for each pair of vertices $\vev$ and $\vew$, the minimal face containing $\vev$ and $\vew$ is a parallelotope with dimension at most $d$, and dimension $d$ is achieved for a minimal face of some pair of vertices.
\end{defn}

An $(n,n)$-parallelotope is an $n$-parallelotope, an $(n, 1)$-parallelotope is a simplex, and an $(n, 0)$-parallelotope is just a point in $\R^n$. See \autoref{figure:nd-parallelotopes} for a visualization of  $(3, d)$-parallelotopes. 

\begin{figure}
    \centering
    \begin{subfigure}[b]{0.3\textwidth}
        \centering
        \begin{tikzpicture}[scale=1.5, inner sep=.75mm, minicirc/.style={circle,draw=black,fill=black,thick}]

            \node (circ1) at (0,0,0) [minicirc] {};
            \node (circ2) at (1,0,0) [minicirc] {};
            \node (circ3) at (0,1,0) [minicirc] {};
            \node (circ4) at (0,0,1) [minicirc] {};

            \draw[black, thick, dashed] (circ1) -- (circ2);
            \draw[black, thick, dashed] (circ1) -- (circ3);
            \draw[black, thick, dashed] (circ1) -- (circ4);
            \draw[black, thick] (circ2) -- (circ3);
            \draw[black, thick] (circ2) -- (circ4);
            \draw[black, thick] (circ3) -- (circ4);

        \end{tikzpicture}
    \end{subfigure}
    \hfill
    \begin{subfigure}[b]{0.3\textwidth}
        \centering
        \begin{tikzpicture}[scale=1.5, inner sep=.75mm, minicirc/.style={circle,draw=black,fill=black,thick}]
            \node (circ1) at (0,0,0) [minicirc] {};
            \node (circ2) at (2,0,0) [minicirc] {};
            \node (circ3) at (1,1,0) [minicirc] {};
            \node (circ4) at (-.5,0,1) [minicirc] {};
            \node (circ5) at (1.5,0,1) [minicirc] {};
            \node (circ6) at (0.5,1,1) [minicirc] {};

            \draw[black, thick, dashed] (circ1) -- (circ2);
            \draw[black, thick, dashed] (circ1) -- (circ3);
            \draw[black, thick] (circ2) -- (circ3);
            \draw[black, thick] (circ4) -- (circ5);
            \draw[black, thick] (circ4) -- (circ6);
            \draw[black, thick] (circ5) -- (circ6);
            \draw[black, thick, dashed] (circ1) -- (circ4);
            \draw[black, thick] (circ2) -- (circ5);
            \draw[black, thick] (circ3) -- (circ6);
        \end{tikzpicture}
    \end{subfigure}
    \hfill
    \begin{subfigure}[b]{0.3\textwidth}
        \centering
        \begin{tikzpicture}[scale=1.5, inner sep=.75mm, minicirc/.style={circle,draw=black,fill=black,thick}]
            \node (circ1) at (0,0,0) [minicirc] {};
            \node (circ2) at (1,0,0) [minicirc] {};
            \node (circ3) at (0,1,0) [minicirc] {};
            \node (circ4) at (0,0,1) [minicirc] {};
            \node (circ5) at (1,1,0) [minicirc] {};
            \node (circ6) at (0,1,1) [minicirc] {};
            \node (circ7) at (1,0,1) [minicirc] {};
            \node (circ8) at (1,1,1) [minicirc] {};

            \draw[black, thick, dashed] (circ1) -- (circ2);
            \draw[black, thick, dashed] (circ1) -- (circ3);
            \draw[black, thick, dashed] (circ1) -- (circ4);
            \draw[black, thick] (circ2) -- (circ5);
            \draw[black, thick] (circ2) -- (circ7);
            \draw[black, thick] (circ3) -- (circ5);
            \draw[black, thick] (circ3) -- (circ6);
            \draw[black, thick] (circ4) -- (circ6);
            \draw[black, thick] (circ4) -- (circ7);
            \draw[black, thick] (circ5) -- (circ8);
            \draw[black, thick] (circ6) -- (circ8);
            \draw[black, thick] (circ7) -- (circ8);
            
        \end{tikzpicture}
    \end{subfigure}
    
    \caption{Examples of $(3,d)$-parallelotopes. From left to right: $d=1$, $d=2$, $d=3$.}
    \label{figure:nd-parallelotopes}
\end{figure}
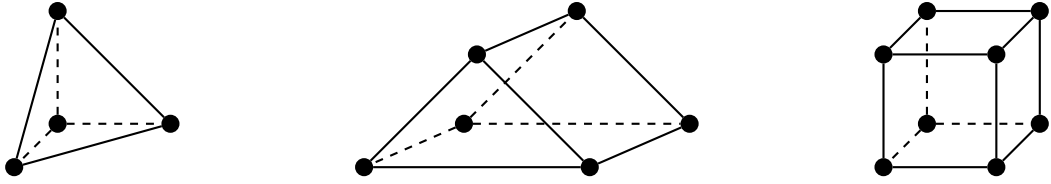

Every maximal circuit walk on an $(n,d)$-parallelotope is an edge walk \cite{bv-17}. A proof of this fact uses the inner cone: let $P$ be a pointed polyhedron and $\vev \in P$ be a vertex, the \emph{inner cone} of $\vev$, $I(\vev)$, is the cone whose extreme rays are the edge directions of edges incident to $\vev$. Given another vertex $\vew \in P$, we define the set $I^{\vev \vew}(\vev)$ to be the inner cone of $\vev$ restricted to the minimal face containing both $\vev$ and $\vew$. The class  of $(n,d)$-parallelotopes is characterized in terms of their inner cones via the so-called {\em symmetric inner cone condition}.

For the remainder, we consider \emph{simple polytopes}. Because vertices of simple polytopes in dimension $d$ have exactly $d$ incident facets, we are able to provide a characterization of each inner cone. \\

\begin{prop}[\cite{bv-17}, Symmetric Inner Cone Condition]\label{prop:inner_cone}
    Let $P$ be a simple polytope. The following statements are equivalent. 
    \begin{enumerate}
        \item $P$ is an $(n,d)$-parallelotope.
        \item All circuit walks in $P$ are edge walks.
        \item $I^{\vev \vew}(\vev) = -I^{\vev \vew}(\vew)$ for all pairs of vertices $\vev, \vew \in P$. 
    \end{enumerate}
\end{prop}

Inner cones of $(n,d)$-parallelotopes are intimately connected with sign-compatible circuit walks on $(n,d)$-parallelotopes. First, we recall the notion of an elementary cone. Let $B \in \R^{m\times n}$ and $P = \{ \vex \in \R^n : B\vex \leq \ved \}$. We define the following hyperplane arrangement $\mathcal{H}(B) := \bigcup_{i=1}^m \{ \vex : B^\T_i\vex = 0 \}$. The set $\mathcal{H}(B)$ partitions $\R^n$ into $n$-dimensional cones with disjoint interiors. The inclusion-minimal $n$-dimensional cones are called \emph{elementary cones}. Suppose that $\vev$ and $\vew$ are vertices of $P$. Let $\mathcal{E}$ be the (intersection of) elementary cone(s) containing $\vew - \vev$. Every sign-compatible circuit walk from $\vev$ to $\vew$ is contained in $(\vev + \mathcal{E}) \cap (\vew - \mathcal{E})$. The elementary cones of an $(n,d)$-parallelotope $P$ are highly structured: the inner cones $I(\vev)$ of $P$ are precisely the elementary cones of $P$ \cite{bv-17}. Therefore, if $P$ is an $(n,d)$-parallelotope, we have $(\vev + \mathcal{E}) \cap (\vew - \mathcal{E}) = (\vev + I^{\vev \vew}(\vev)) \cap (\vew - I^{\vev \vew}(\vev))$.

First, we show that if $P$ is an $(n,d)$-parallelotope, then there exists a sign-compatible circuit walk between every pair of vertices of $P$. Then, we show that, under some additional restrictions, $(n,d)$-parallelotopes are the unique colletion of polytopes satisfying those restrictions such that there exists a sign-compatible circuit walk between every pair of vertices.

\begin{lem}
    Let $P$ be a pointed polyhedron and $\vev$, $\vew \in P$ be vertices. If $P$ is an $(n,d)$-parallelotope, then there exists a sign-compatible circuit walk from $\vev$ to $\vew$.
\end{lem}

\begin{proof}
    Let $\vev$, $\vew \in P$ be vertices. Let $F$ be the minimal face containing $\vev$ and $\vew$ and note that $F$ is a $\operatorname{dim}(F)$-parallelotope. We regard $F$ as a polyhedron with vertices $\vev$ and $\vew$. Note that every circuit of $F$ is a circuit of $P$ and every circuit walk in $F$ is a circuit walk in $P$. Thus, if we show that there exists a sign-compatible circuit walk from $\vev$ to $\vew$ in $F$, then we have proved the claim. 

    Let $\mathcal{E}$ be the elementary cone (of $F$) containing $\vew-\vev$. There exists a sign-compatible circuit decomposition of $\vew - \vev$ that is an edge walk on the spindle $(\vev + \mathcal{E}) \cap (\vew - \mathcal{E})$. From \cite{bv-17}, $\mathcal{E} = I(\vev) = -I(\vew)$. Because $F$ is a parallelotope, $F = (\vev + I(\vev)) \cap (\vew + I(\vew))$. Thus, the sign-compatible circuit decomposition is an edge walk in $F$, and therefore, it is a maximal circuit walk in $F$. 
\end{proof}

Recall that a sign-compatible circuit walk between a pair of vertices necessarily has length at most the dimension of the minimal face containing the pair of vertices. We show that under some assumptions on the length of each sign-compatible walk, all polyhedra that have sign-compatible circuit walks between all pairs of vertices are $(n, \, d)$-parallelotopes.

\begin{thm}\label{thm:some_scm}
    Let $P$ be a simple, pointed polyhedron. 
    If for all pairs of vertices $\vev$, $\vew \in P$, each sign-compatible circuit decomposition of $\vew-\vev$ gives a sign-compatible circuit walk of length $d$, where $d$ is the dimension of the minimal face containing $\vev$ and $\vew$,
    then $P$ is an $(n,d^*)$-parallelotope, where $d^*$ is the maximum of $d$ over all choices of $\vev$ and $\vew$. 
\end{thm}

\begin{proof}
    Let $F$ be the minimal face containing $\vev$ and $\vew$. We will show that $F$ is a parallelotope. Let $((\veg_1, \lambda_1), \ldots, (\veg_d, \lambda_d))$ denote a sign-compatible circuit walk from $\vev$ to $\vew$ of length $d$. We assume (without loss of generality) that each $\veg_i$ is scaled to be maximal, that is, $\lambda_i = 1$. 

    As $P$ is simple, $F$ is simple. Because $((\veg_1, \lambda_1), \ldots, (\veg_d, \lambda_d))$ is a sign-compatible circuit walk from $\vev$ to $\vew$, it follows that $\vev + \sum_{i=1}^k \lambda_i \veg_i \in F$ for each $1 \leq k \leq d$. This implies that $((\veg_1, \lambda_1), \ldots, (\veg_d, \lambda_d))$ is a sign-compatible circuit walk from $\vev$ to $\vew$ in $F$. Let $F_1, \ldots, F_{k_1}$ be the facets of $F$ incident to $\vev$ and $G_1, \ldots, G_{k_2}$ be the facets of $F$ incident to $\vew$. By simplicity of $F$, $k_1 = k_2 = d$.  By minimality of $F$, $F_i \neq G_j$ for any pair $i,j$. 
    
    We claim that each $\veg_i$ is parallel to exactly $d-1$ facets from among the collection $\{ F_1, \ldots, F_d \}$ and parallel to exactly $d-1$ facets from among the collection $\{ G_1, \ldots, G_d \}$. To see this, note that by maximality, $\vev + \veg^1$ is contained in some facet of $F$ that is not incident to $\vev$ and by sign-compatibility, this facet is $G_i$ for some $1\leq i \leq d$. Without loss of generality, we assume $G_i = G_1$. By sign-compatibility, each $\veg_2, \ldots, \veg_d$ is parallel to $G_1$. Because rearranging the steps of a sign-compatible circuit decomposition still yields a sign-compatible circuit decomposition and because each rearrangement is a sign-compatible circuit walk, we have that each $\veg_1, \ldots, \veg_{i-1}, \veg_{i+1}, \ldots, \veg_d$ is parallel to $G_i$.  
    
     The reversed decomposition $((-\veg_1, 1), \ldots, (-\veg_d, 1))$ is a sign-compatible circuit decomposition of $\vev - \vew$. By assumption, it is a sign-compatible circuit walk from $\vew$ to $\vev$. By similar arguments, for each $i$, the circuits $-\veg_1, \ldots, -\veg_{i-1}, -\veg_{i+1}, \ldots, -\veg_d$ (and thus, $\veg_1, \ldots, \veg_{i-1}, \veg_{i+1}, \ldots, \veg_d$) are parallel to $F_i$. Thus, each $\veg_i$ is parallel to exactly $d-1$ facets of $F_1, \ldots, F_d$ and parallel to exactly $d-1$ facets of $G_1, \ldots, G_d$.

    We claim that each $\veg_i$ is an edge direction of an edge incident to $\vev$ in $F$. For each $i$, $\vev + \veg_i \in G_i \cup \bigcap_{j \neq i} F_j$. By simplicity, each vertex is the intersection of $d$ facets and two vertices are adjacent if they are contained in the intersection of $d-1$ facets. It follows that $\vev + \veg_i$ is a vertex adjacent to $\vev$ because $\veg_i$ is parallel to $d-1$ facets and enters a new facet. Therefore, $\veg_i$ is an edge direction of an edge incident to $\vev$.
    Because there are $d$ edges incident to $\vev$ and each $\veg_i$ is distinct, it follows that each edge direction of an edge incident to $\vev$ corresponds to some $\veg_i$. Let $\mathcal{E}$ be the elementary cone containing $\vew - \vev$. Clearly, $I(\vev) \subseteq \vev+\mathcal{E}$. However, because $\mathcal{E}$ is an elementary cone, we have $\vev + \mathcal{E} \subseteq I(\vev)$, and therefore, $\vev + \mathcal{E} = I(\vev)$. By similar arguments, $\vew - \mathcal{E} = I(\vew)$. Therefore, $F$ is a parallelotope and $P$ is an $(n,d^*)$-parallelotope.
\end{proof}

\autoref{thm:some_scm} shows that, under some assumptions on the length of each sign-compatible walk, the class of polytopes that admit a sign-compatible circuit walk between every pair of vertices are $(n,d)$-parallelotopes. In the conclusion, we discuss a few directions towards classifying all polyhedra that admit a sign-compatible circuit walk between every pair of vertices.  

\section{Conclusion}\label{sec:conclusion}

We resolved some fundamental questions about the complexity of short circuit walks in polyhedra, but many interesting questions remain open. First, we demonstrated that it is NP-hard to determine the circuit distance between a pair of vertices. This hardness holds even for circulation polytopes with unit capacities, which are special $0/1$-TU polytopes. However, it remains open whether it is NP-hard to determine the circuit {\em diameter}. For combinatorial diameters, weak and strong NP-hardness is known \cite{ft-94,s-18}, even for the highly-structured fractional matching polytope \cite{s-18}. Our proof crucially relies on knowledge of a starting and a target vertex, and does not readily provide an approach to this question: it could be possible that determining the circuit diameter is efficiently possible even though determining the distance between a pair of given vertices is not. We believe it is hard. 

We transferred the proof to also observe hardness of determining the distance under the restriction of sign-compatibility of the circuit steps, with or without a requirement of maximal step lengths. In particular, this implies that for a $d$-dimensional polytope $P$ and a pair of vertices $\vev, \vew \in P$, it is NP-hard to determine if $\vew-\vev$ can be decomposed into a conformal sum using two circuits. In contrast, it is well-known that $\vew-\vev$ can always be decomposed into a conformal sum of at most $d$ circuits \cite{r-69}. This leaves a range of $3 \leq k \leq d-1$, for which it remains open whether it is hard to decide whether we can decompose $\vew-\vev$ into a conformal sum of $k$ circuits. Another consequence is that it is NP-hard to determine the minimum total length needed for a circuit walk between a pair of vertices for the $p$-norms with $p>1$. We are interested in identifying settings where circuit walks of minimal total length, or provable approximations, can be constructed efficiently. 

The above results and questions relate to a shortest circuit walk in the sense of few steps or small geometric length. Our work in Sections \ref{sec:single_step_hard} and \ref{sec:TU_easy} relates to the desirable property of `picking up correct facets,' i.e., finding a step that is contained in a facet incident to the target.  A walk which achieves this at each step has length at most $d$. We proved that the decision whether a first step can satisfy this property is already NP-hard. However, our proof does not resolve the equally interesting question about the existence of a sign-compatible maximal step, which would necessarily  pick up a correct facet and \textit{also} exhibit additional favorable features that give rise to short circuit walks. Our proof does not resolve this because we check only for containment in a facet, but not sign-compatibility overall. On the other hand, we do prove that in the special case of TU polyhedra this decision is efficient and we are interested in the identification of additional settings where the same holds. Specifically, we would like to understand whether the efficiency breaks down as soon as the components of the circuits have more than $3$ possible values.

A natural generalization is to not only consider the first step of a walk, but a complete walk which enters a correct facet in each step. This in turn leads to a desire to understand which classes of polytopes are guaranteed to have sign-compatible walks. This question can be interpreted in several ways. In Section \ref{sec:alwayssc}, we resolve a particularly strong interpretation for simple polytopes: if, for all pairs of vertices, all sign-compatible circuit decompositions give sign-compatible circuit walks with a number of steps equal to the dimension of the shared face, then the polytope must be some $(n,d)$-parallelotope. We are interested in a generalization to degenerate polytopes. Further, this only gives a partial classification of simple polytopes with sign-compatible circuit walks, and we leave open two other cases.

The first case is when only some sign-compatible circuit decompositions are circuit walks. After reordering the steps, the sign-compatible circuit decomposition may not be a circuit walk. The second case is when there exist ``shortcuts:'' a sign-compatible circuit decomposition may have fewer than $d$ circuits, and in such a case individual steps enter several facets. Weaker interpretations of the original question, like the validity of the property for only a subset or even just one pair of vertices, lead to further interesting topics on the geometric characterization of the classes of corresponding polytopes. 

A particularly good starting point for research in this direction could be the consideration of polytopes of circuit diameter $1$, which clearly always admit a sign-compatible circuit walk between any pair of vertices: if $v$ and $w$ are two vertices then $w-v$ is a circuit and is necessarily a sign-compatible, maximal step. Several well-known classes of polytopes from combinatorial optimization are known to have circuit diameter $1$, such as the traveling salesman polytope on the complete graph $K_n$ ($n \neq 5$) and the perfect matching polytope on $K_n$ ($n \neq 8$) \cite{kps-17}. 

A classification of \textit{all} polytopes with circuit diameter $1$ is a challenging and interesting problem in its own right. The corresponding task of a characterization of all polytopes with combinatorial diameter 1, also known as $2$-\textit{neighborly} polytopes, is a well-established field of study. Even for low dimensions and when restricting to $0/1$-polytopes, a study of their combinatorial properties is extremely challenging: in dimension $7$, there are over $13$ billion equivalence classes \cite{m-19}. In the circuit setting, additional challenges arise already in dimension $2$. While combinatorial diameters only depend on the topology of the skeleton, circuit diameters depend on its realization. For example, while no quadrilateral has circuit diameter $1$, a regular pentagon does have circuit diameter $1$ and perturbing even just one edge causes the diameter to increase to $2$. 

\section*{Acknowledgments}

S. Borgwardt, J. Lee and W. Grewe were partially supported by (US) Air Force Office of Scientific Research grant FA9550-21-1-0233 (Complex Networks). S. Borgwardt was partially supported by NSF grant 2006183 (Algorithmic Foundations, Division of Computing and Communication Foundations). S. Kafer was partially supported by the NSF under Grant No. DMS-1929284 while the author was in residence at the Institute for Computational and Experimental Research in Mathematics in Providence, RI, during the Discrete Optimization: Mathematics,  Algorithms, and Computation semester program. L. Sanit\`a was partially supported by the NWO VIDI grant VI.Vidi.193.087.

\bibliographystyle{plain}

\end{document}